\documentclass[a4paper, reqno]{amsart}

\usepackage[english]{babel}
\usepackage{amsmath,amssymb,amsfonts,dsfont}
\numberwithin{equation}{section}
\usepackage{lmodern}
\usepackage{fullpage}
\usepackage{csquotes}

\usepackage{amsmath,amsthm,mathtools}
\usepackage{graphicx}
\usepackage[dvipsnames,table,xcdraw]{xcolor}
\definecolor{myred}{RGB}{228,26,28}
\definecolor{myblue}{RGB}{55,124,184}
\definecolor{mygreen}{RGB}{77,175,74}
\definecolor{myorange}{RGB}{255,127,14}
\definecolor{myviolet}{RGB}{148,103,189}

\usepackage[colorinlistoftodos]{todonotes}
\usepackage[colorlinks=true,urlcolor=myblue,linkcolor=myblue,runcolor=myred,citecolor=myred,linktoc=all]{hyperref}
\usepackage{caption}
\usepackage{enumitem}
\setlist[itemize]{label={$\vcenter{\hbox{\tiny$\bullet$}}$}}
\usepackage[boxruled, lined]{algorithm2e}

\usepackage{subfigure}

\usepackage{tikz, pgf}
\usetikzlibrary{matrix, calc, backgrounds, arrows.meta}
\usepackage{tikz-3dplot}
\usepackage{tkz-euclide}
\usepackage{pgfplots}
\usepackage{pgfplotstable}
\pgfplotsset{compat=1.8}
\usepackage{booktabs}

\newcommand{\Forall}{\forall\ }

\def\R{{\mathbb R}}
\def\N{{\mathbb N}}
\def\Ran{{\text{Ran}}}

\newcommand{\Hc}{\mathcal{H}}
\renewcommand{\Mc}{\mathcal{M}}
\newcommand{\Tc}{T}
\newcommand{\Lc}{\mathcal{L}}
\newcommand{\Cc}{\mathcal{C}}

\newcommand{\Id}{\text{Id}}

\newcommand{\oo}[1]{\prt{#1}_\text{oo}}
\newcommand{\ov}[1]{\prt{#1}_\text{ov}}
\newcommand{\vo}[1]{\prt{#1}_\text{vo}}
\newcommand{\vv}[1]{\prt{#1}_\text{vv}}

\newcommand{\gpp}{\ddot\gamma}
\newcommand{\Op}{{\Omega_*}}
\newcommand{\auf}{\Phi}

\newcommand{\hc}{\text{sym}}
\DeclareMathOperator*{\Tr}{Tr}
\newcommand{\norm}[1]{\left\lVert#1\right\rVert}
\newcommand{\normF}[1]{\left\lVert#1\right\rVert_\text{F}}

\renewcommand{\d}{\text{d}}
\newcommand{\prt}[1]{\left( #1 \right)}
\newcommand{\set}[1]{\left\{ #1 \right\}}
\newcommand{\cro}[1]{\left\langle  #1 \right\rangle}
\newcommand{\croF}[1]{\left\langle  #1 \right\rangle_\text{F}}

\newcommand{\ie}{\emph{i.e.}\ }
\renewcommand{\i}{\text{i}}
\renewcommand{\lq}{\leqslant}
\renewcommand{\le}{\leqslant}
\renewcommand{\ge}{\geqslant}
\newcommand{\gq}{\geqslant}

\theoremstyle{plain}
\newtheorem{theorem}{Theorem}[section]

\newtheorem{lemma}[theorem]{Lemma}

\newtheorem{proposition}[theorem]{Proposition}
\newtheorem{asmp}[theorem]{Assumption}
\theoremstyle{remark}
\newtheorem{remark}[theorem]{Remark}
\newtheorem{expl}[theorem]{Example}

\addto\extrasenglish{%
}
\addto\extrasenglish{%
}

\title{Convergence analysis of direct minimization and self-consistent iterations}
\author{Eric Canc\`es, Gaspard Kemlin, Antoine Levitt}
\date{\today}

\begin{document}
\maketitle

\begin{abstract} This article is concerned with the numerical solution
  of subspace optimization problems, consisting of minimizing a smooth
  functional over the set of orthogonal projectors of fixed rank. Such
  problems are encountered in particular in electronic structure
  calculation (Hartree-Fock and Kohn-Sham Density Functional Theory -- DFT -- models). We compare from a
  numerical analysis perspective two simple representatives, the damped self-consistent field (SCF) iterations and the gradient descent algorithm,
  of the two classes of methods competing in the field: SCF and direct
  minimization methods. We derive asymptotic rates of convergence for
  these algorithms and analyze their dependence on the spectral gap
  and other properties of the problem. Our theoretical results are
  complemented by numerical simulations on a variety of examples, from
  toy models with tunable parameters to realistic Kohn-Sham
  computations. We also provide an example of chaotic behavior of the
  simple SCF iterations for a nonquadratic functional.
\end{abstract}

\section{Introduction}
This paper is concerned with the convergence behavior of algorithms to
solve the \textit{subspace optimization problem}
\begin{equation}
  \label{eq:problem_with_trace}
  \min \Big\{E(P) \;\Big|\; P\in\R^{N_b\times N_b},\ P^2 = P = P^*,\ \Tr(P)=N\Big\}
\end{equation}
consisting of optimizing a $C^2$ function $E: \R^{N_b\times N_b} \to \R$ over the set of rank-$N$ orthogonal projectors
$P$. Here $P^*$ denotes the adjoint (transpose) of $P$. This problem can
also be reformulated as
\begin{align} \label{eq:problem_orbitals}
  \min\set{E\left(\sum_{i=1}^{N} \phi_{i} \phi_{i}^{*}\right) \;\middle|\;
    \phi_{i} \in \R^{N_{b}},\  \phi_{i}^{*} \phi_{j} = \delta_{ij} \quad \Forall
    i,j \in \{1,\dots,N\}},
\end{align}
using an orthonormal basis $(\phi_{i})_{i=1,\dots,N}$ for the subspace
$\Ran (P)$. This problem is of interest in a number of contexts, such
as matrix approximation, computer vision
\cite{absilOptimizationAlgorithmsMatrix2008}, and electronic structure
theory~\cite{CDKLM2003,bookHelgaker,LinLu,linNumericalMethodsKohn2019b,martin2004,saadNumericalMethodsElectronic2010},
the latter of which being the main motivation for this work.

Let $H(P) = \nabla E(P)$. The first order conditions for
problem \eqref{eq:problem_with_trace} is
\begin{align*}
  P H(P)(1-P) = (1-P) H(P) P = 0.
\end{align*}
Up to an appropriate choice for the orthonormal basis $(\phi_{i})_{i=1,\dots,N}$
of $\Ran(P)$, this yields
\begin{align}
  \label{eq:nonlinear_eigenvector_problem}
  H(P) \phi_{i} = \varepsilon_{i} \phi_{i},
\end{align}
which reveals an alternative interpretation of this problem as a
\textit{nonlinear eigenvector problem} (to be distinguished from
\textit{nonlinear eigenvalue problems} of the form
$A(\varepsilon) \phi = 0$, where $A : \R \to \R^{N_b \times N_b}$). In
the case when $E(P) = \Tr(H_{0} P)$ for a fixed symmetric matrix
$H_{0}$, one recovers the classical eigenvalue problem
$H_{0} \phi_{i} = \varepsilon_{i} \phi_{i}$. At a minimizer of
\eqref{eq:problem_with_trace}, the
$(\varepsilon_{i})_{i = 1,\dots, N}$ are the lowest eigenvalues of
$H_{0}$, counting multiplicities.

Problems of the form \eqref{eq:problem_with_trace} are found in the
Hartree-Fock and Kohn-Sham theories of electronic
structure~\cite{bookHelgaker,martin2004}, both approximations of the
many-body Schr\"odinger equation. In this context, the $\phi_{i}$ are
(discretized) \textit{orbitals}, the projector $P$ is the
\textit{density matrix}, and the energy $E(P)$ includes linear
contributions from the kinetic and external potential energy of the
electrons, as well as nonlinear terms arising from electron-electron
interaction. Another notable problem of this form is the nonlinear
Schr\"odinger or Gross-Pitaevskii equation for Bose-Einstein
condensates~\cite{Bao2013}, where $N=1$. In all these cases, the
first-order condition \eqref{eq:nonlinear_eigenvector_problem} is
interpreted as a \textit{self-consistent} or \textit{mean-field}
equation: the particles behave as independent particles in an
effective Hamiltonian $H(P)$ (also known as the Fock matrix) involving
the mean-field they create. In the rest of this paper, we will work on
the formulation \eqref{eq:problem_with_trace} without specifying $E$
for generality.

The minimization problem \eqref{eq:problem_with_trace} is compact but
nonconvex: there exists at least one minimizer, but the minimizer might not be
unique, and local minima might not be global ones. Solving this
optimization problem is of considerable practical interest, and
algorithms for doing so date back to the early days of quantum
mechanics~\cite{hartree_1928}. The first introduced and still most popular approach is the
\textit{self-consistent field} (SCF) method, which, in its original version~\cite{McWeeny1956,Roothaan1951}, works as follows: if $P^{k}$ is the current
iterate of the algorithm, $P^{k+1}$ is found by solving
\eqref{eq:nonlinear_eigenvector_problem} for the fixed matrix $H(P^k)$:
\begin{align*}
  H(P^{k}) \phi_{i}^{k} = \varepsilon_{i}^{k} \phi_{i}^{k},\quad (\phi_{i}^{k})^{*}\phi_{j}^{k} = \delta_{ij}
\end{align*}
with the $\varepsilon_{i}^{k}$ sorted in non-decreasing order, and
building $P^{k+1}$ as
\begin{align*}
  P^{k+1} = \sum_{i=1}^{N} \phi_{i}^{k} (\phi_{i}^{k})^{*}.
\end{align*}
This algorithm assumes the \textit{Aufbau} property, which is that at
a minimum $P_*$ we have $P_* = \sum_{i=1}^{N} \phi_{i} \phi_{i}^{*}$ with
$\phi_{i}$ a system of orthogonal eigenvectors associated with the
lowest $N$ eigenvalues of $H(P_*)$. This property holds for the
(spin-unconstrained) Hartree-Fock model \cite{Bach1994} and the
Gross-Pitaevskii models without magnetic field
\cite{cances2010numerical}, usually holds for molecular systems in the
Kohn-Sham model, but does not hold in general for Gross-Pitaevskii models with
strong magnetic fields.

This basic procedure converges for
systems where the nonlinearity is weak, but fails to converge
otherwise~(see \cite{cancesConvergenceSCFAlgorithms2000c} for a
comprehensive mathematical analysis of this behavior when the
functional $E$ is a sum of a linear and a quadratic term in $P$, which
is the case for the Hartree-Fock model). A solution is to
\textit{damp} this procedure, and {\em mix} the iterates to accelerate
convergence. This gives rise to a variety of SCF algorithms, among
which Broyden-like and Anderson-like mixing
algorithms~\cite{Johnson1988,Marks2008,Raczkowski2001,Srivastava1984},
the Direct Inversion in the Iterative Space (DIIS)
algorithm~\cite{Kresse1996,Pulay1980,Pulay1982}, the Optimal Damping
Algorithm \cite{cancesCanWeOutperform2000b} (ODA), and the Energy-DIIS
(EDIIS) algorithm combining the latter two
approaches~\cite{Kudin2002}.

A second class of algorithms solves the minimization problem
\eqref{eq:problem_with_trace} directly. The minimization set
$\{P \in \R^{N_{b} \times N_{b}},\ P^{2} = P^{*} = P,\ \Tr P = N\}$ is
diffeomorphic to the Grassmann manifold of the $N$-dimensional vector
subspaces of $\R^{N_b}$. This set is naturally equipped with the
structure of a Riemannian manifold, and this allows the use of Riemann
optimization
algorithms~\cite{absilOptimizationAlgorithmsMatrix2008,Edelman1998}.
Direct minimization algorithms are preferred for the Gross-Pitaevskii
model with magnetic
fields~\cite{ANTOINE201792,Danaila2017,heid2019gradient,henning2018sobolev},
for which the {\em Aufbau} principle is not satisfied in general.
Gradient-type~\cite{Alouges2009,Zhuang,Mostofi2003,VYP,ZZWZ},
Newton-type \cite{Bacskay1981,Chaban1997,ZBJ}, and trust-region
methods have also been designed to solve~\eqref{eq:problem_with_trace}
for larger values of $N$. At the time of writing, direct minimization
algorithms are less popular than SCF algorithms in electronic
structure calculation, where $N$ can be very large, but it is not
clear whether this is for sound scientific reasons or because SCF
algorithms have been implemented and optimized for decades in the main
production codes, which has not been the case for direct minimization
algorithms.

While the convergence of several SCF and direct minimization
algorithms has been analyzed from a mathematical point of view (see
e.g.~\cite{chupin2020convergence,levittConvergenceGradientbasedAlgorithms2012b,liuConvergenceSelfConsistentField2013a,Rohwedder2011,upadhyayaDensityMatrixApproach2018a,yangConvergenceSelfConsistentField2009}
and references therein), the two approaches have not been compared in
a systematic way to our knowledge. The purpose of this paper is to
contribute to fill this gap, by focusing on very simple
representatives of each class, namely the damped SCF iteration and the
gradient descent. We emphasize that neither of these two algorithms is
a practical choice as is. The SCF iteration should be accelerated (for
instance using the Anderson acceleration technique), and the gradient
information in direct minimization methods should rather be used as
part of a quasi-Newton method (such as the limited-memory BFGS
algorithm \cite{absilOptimizationAlgorithmsMatrix2008}). Depending on
the exact problem at hand, all these methods should be preconditioned
to avoid issues related to small mesh sizes (which leads to a
divergence of the kinetic energy term) and/or large computational
domains (which can lead to a divergence of the Coulomb energy, or the
confining potential). We refer to \cite{woodsComputingSelfconsistentField2019} for a
recent review in the context of the Kohn-Sham equations for solids.
Rather, in this paper, we aim to focus on the very simplest
representative of each general strategy (SCF and direct minimization).
The investigation of these two basic algorithms is informative on the
strengths and weaknesses of the two classes, and is a first step in
the analysis of more complex methods.

The paper is organized as follows.
In Section~\ref{sec:optim}, we recall some results about optimization on Grassmann manifolds, in
particular the first and second order optimality conditions, and prove
preparatory lemmas. In Section~\ref{sec:algo}, we present the two algorithms that are in the
scope of this paper: a fixed-step gradient descent and a damped SCF
algorithm. We prove their local convergence as long as the step is
small enough and we derive convergence rates. We find that the
convergence rates depend on the spectral radius of operators (acting
on $\R^{N_{b} \times N_{b}}$) of the form $1 - \beta J$, with $\beta$
the fixed step and $J = \Op + K_*$ for the gradient descent,
$J = 1 + \Op^{-1}K_*$ for the SCF algorithm, where the operators $\Op$
and $K_*$ are specified in the next section. Let us just mention at
this stage that the lowest eigenvalue of $\Op$ is equal to the
spectral gap between the $N^{\rm th}$ and $(N+1)^{\rm st}$ eigenvalues
of $H(P_*)$, allowing us to analyze the convergence rates of the algorithms
in terms of natural quantities of the problem. This also shows that
the damped SCF algorithm can be seen as a matrix splitting of the
fixed-step gradient descent algorithm.

In Section~\ref{sec:num_res}, we compare the two algorithms on
several test problems. First, we focus on a toy model for which we can
easily tune the gap and observe some fundamental differences between
SCF and direct minimization algorithms, in agreement with the
mathematical results established in Section~\ref{sec:algo}. We also
provide an example of chaos in SCF iterations, complementing the
results of
\cite{cancesConvergenceSCFAlgorithms2000c,levittConvergenceGradientbasedAlgorithms2012b}
in the case of a non-quadratic objective functional $E$. Then, we
analyse a 1D Gross-Pitaevskii model ($N=1)$ and its fermionic
counterpart for $N=2$, for which we investigate the behavior of the
algorithms when the gap closes. We conclude with an example from
electronic structure calculation: a Silicon crystal, in the framework
of Kohn-Sham DFT, where we show in particular that accelerated SCF
algorithms are less sensitive to small gaps than the simple damped
SCF. Finally, in Section \ref{sec:conclusion} we draw conclusions and
outline perspectives for future work.

\section{Optimization on  Grassmann manifolds}
\label{sec:optim}

We focus in this paper on the case of real symmetric matrices, but the study
can be easily extended to complex hermitian matrices.
Let $\Hc \coloneqq \R^{N_b\times N_b}_\text{sym}$ be the vector space of
$N_b \times N_b$ real symmetric matrices endowed with the Frobenius inner
product $\croF{A,B} \coloneqq \Tr(AB)$.
Let
\[ \Mc \coloneqq \set{P\in\Hc \;\middle|\; P^2=P} \text{ and }
\Mc_N \coloneqq \set{P\in\Hc \;\middle|\; P^2=P,\ \Tr(P)=N}.\]
From a geometrical point of view, $\Mc$ is a compact subset of $\Hc$ with $N_b+1$
connected components $\Mc_N$, $N = 0,\dots,N_b$,
each of them being characterized by the value of $\Tr(P)$,
namely the rank of the orthogonal projector $P$, and being diffeomorphic to
the Grassmann manifold $\text{Grass}(N,N_b)$
\cite{absilOptimizationAlgorithmsMatrix2008}.
From now on, we fix
the number of electrons $N$ and we seek the local minimizers of the problem
\begin{equation}
  \label{mini_mani}
  \min_{P\in\Mc_N} E(P),
\end{equation}
where $E : \Hc \to \R$ is a discretized energy functional, for which some examples are given
below.

\begin{expl}
  As an example, we study a discrete Gross-Pitaevskii model in
  \autoref{sec:validate}. Other models from electronic structure
  can be considered, such as the discretized
  Hartree-Fock or Kohn-Sham models, where the energy is of the form
  \[
    E(P) \coloneqq \Tr(H_0P) + E_\text{nl}(P)
  \]
  with $H_0$ being the core Hamiltonian (representing the kinetic energy and the
  external potential) and $E_\text{nl}$ a nonlinear energy functional
  depending on the model (representing the interaction between electrons).
  For instance, for the Hartree-Fock model,
  \[
    E_\text{nl}(P) \coloneqq \frac{1}{2}\Tr(G(P)P) \qquad \text{where} \qquad
    (G(P))_{ij} \coloneqq \sum_{k,l = 1}^{N_b} A_{ijkl}P_{kl} \quad \Forall i,j=1,\dots,N_b,
  \]
  with $A$ a symmetric tensor of order 4.
  For more details on these models or electronic structure in general, we
  refer to \cite{CDKLM2003,linNumericalMethodsKohn2019b,saadNumericalMethodsElectronic2010}.
\end{expl}

In plane-wave, finite differences, finite elements or wavelets electronic
structure calculation codes, the size $N_b$
of the discretized space is in practice much larger than the number
$N$ of electrons. Therefore, it is not practical to store and
manipulate the (dense) matrix $P$. Instead, algorithms work on
the orbitals $(\phi_i)_{i=1,\dots,N}$ introduced in
\eqref{eq:nonlinear_eigenvector_problem}. The density matrix $P$ is then recovered as
  \[
    P = \sum_{i=1}^N \phi_i\phi_i^*.
  \]
  All the results in this article are presented in the density matrix
framework. However, the algorithms we study can be expressed in a way that
avoids ever forming the density matrix. We refer to
\cite{woodsComputingSelfconsistentField2019} for details.

We will need two assumptions for our results.
\begin{asmp} \label{asmp1}
  The energy functional $E: \Hc \to \R$ is of class $C^2$ (twice continuously differentiable).
\end{asmp}
Assumption~\ref{asmp1} is true for Hartree-Fock models. For Kohn-Sham
models, it is true when the density $\rho = \sum_{i=1}^N |\phi_i|^2$
is uniformly bounded avay from zero,
which is the case for instance in condensed
phase systems. Most of the results presented in this article are local in
nature, and therefore this assumption can be relaxed to local regularity.

\begin{asmp} \label{asmp2} $P_*\in\Mc_N$ is a nondegenerate local minimizer
  of \eqref{mini_mani}
  in the sense that there exists some $\eta > 0$ such
  that, for $P\in\Mc_N$ in a neighborhood of $P_*$, we have
  \[
    E(P) \geqslant E(P_*) + \eta\normF{P-P_*}^2.
  \]
\end{asmp}
It is very hard in most practical situations to check this assumption,
but it seems to be verified in practice. Notable exceptions are
systems invariant with respect to continuous symmetry groups, in
which case $E(P) = E(P_{*})$ for all $P$ in the orbit of $P_{*}$ along
the symmetry group. In this case, the assumption can not be true, and
$\normF{P-P_{*}}$ must be replaced by the distance from $P$ to the
orbit of $P_{*}$. Our results can be extended to this case up to
quotienting $\Hc$ by the symmetry group.

Throughout the paper, we will use the following notation:
\begin{itemize}
  \item $H(P) \coloneqq \nabla E(P)$ is the gradient, and $H_* \coloneqq H(P_*)$;
  \item $K(P) \coloneqq \Pi_P \nabla^2E(P)\Pi_P$ is the Hessian
    projected onto the tangent space at $P$, and $K_* \coloneqq K(P_*)$ (the
    projection $\Pi_P$ is defined below in Proposition~\ref{prop:tangent}).
\end{itemize}

\subsection{First-order condition}
\label{sec:first-order}
To study the first-order optimality conditions, we start by recalling some
classical results about the geometry of the manifold
$\Mc_N$.
\begin{proposition}
  \label{prop:tangent}
  $\Mc_N$ is a smooth real manifold and its tangent space $\Tc_P\Mc_N$ at $P\in\Mc_N$ is given by
  \[
    \Tc_P\Mc_N = \set{X\in\Hc \;\middle|\; PX + XP = X,\ \Tr(X) = 0} =
    \set{X\in\Hc \;\middle|\; PXP = (1-P)X(1-P) = 0}.
  \]
  The orthogonal projection $\Pi_P$ on $\Tc_P\Mc_N$ for the Frobenius inner product is
  \begin{equation}
    \label{eq:prop_proj}
    \Forall X\in\Hc,\quad \Pi_P(X) = PX(1-P) + (1-P)XP = [P,[P,X]],
  \end{equation}
  where $[A,B] \coloneqq AB - BA$.
\end{proposition}
This classical result is proved in e.g.
\cite[Section 3.4]{absilOptimizationAlgorithmsMatrix2008}. Using the fact that
$\Hc = \Ran(P) \oplus \Ran(1-P)$ and the induced decomposition of $P\in\Mc_N$ and $X\in\Hc$ as
\begin{equation}
  \label{eq:XP}
  P =
  \begin{bmatrix}
    I_N & 0 \\ 0 & 0
  \end{bmatrix},\quad
  X =
  \begin{bmatrix}
    \oo{X} & \ov{X} \\ \vo{X} & \vv{X}
  \end{bmatrix},
\end{equation}
the projection $\Pi_P$ is given by
\[
  \Pi_P(X) =
  \begin{bmatrix}
    0 & \ov{X} \\ \vo{X} & 0
  \end{bmatrix}.
\]
Here the subscript \enquote{o} (resp. \enquote{v}) stand for
\emph{occupied} (resp. \emph{virtual}).
The first-order optimality condition at $P_*$ is $\Pi_{P_*}(H_*) = 0$,
which can be formulated as follows:
\begin{equation}
\label{eq:fo-opt}
\boxed{\textbf{First-order optimality condition: } P_*H_*(1-P_*) = (1-P_*)H_*P_* = 0.}
\end{equation}
Note that this condition can be rewritten as $[H_*,P_*] = 0$, showing that
$H_*$ and $P_*$ can be codiagonalized. Let
$(\phi_{k})_{1\leqslant k \leqslant N_b}$ be an orthonormal basis of
eigenvectors of $H_{*}$ associated with the
eigenvalues $(\varepsilon_{k})_{1 \leqslant k \leqslant N_{b}}$ sorted
in ascending order. Then $P = \sum_{i \in \mathcal{I}} \phi_{i}\phi_{i}^{*}$,
where $\mathcal{I} \subset \{1,\dots,N_{b}\}, \,|\mathcal{I}| = N$ is the set of occupied
orbitals. The minimizer $P_*$ is said to satisfy
\begin{itemize}
\item the \emph{Aufbau} principle if $\mathcal{I} = \{1,\dots,N\}$;
\item the strong \emph{Aufbau} principle if
$\mathcal{I} = \{1,\dots,N\}$ and if in addition $\varepsilon_{N} < \varepsilon_{N+1}$, in which case $P_* = \sum_{i=1}^N \phi_i\phi_i^*$.
\end{itemize}

\subsection{Second-order condition}
\label{sec:second-order}
We derive here the second-order optimality condition from the nondegeneracy
of the minimum (Assumption~\ref{asmp2}).

Let $X\in\Tc_{P_*}\Mc_N$, $I$ be a real interval containing $0$
and $\gamma : I\to\Mc_N$ be a smooth path such that
$\gamma(0) = P_*$ and $\dot\gamma(0) = X$. An example of a possible
$\gamma$ is given in Section~\ref{sec:retrac}.
We expand
\[
\begin{split}
  E\prt{\gamma(t)}
  &= E(P_*) + t\croF{H_*,X} + \frac{t^2}{2}\Big(\croF{H_*,\gpp(0)} + \croF{X,
      \nabla^2E(P_*) X}\Big) + o(t^2) \\
  &= E(P_*) + \frac{t^2}{2}\Big(\croF{H_*,\gpp(0)} + \croF{X,
    K_* X}\Big) + o(t^2)
\end{split}
\]
as $H_*$ is orthogonal to $\Tc_{P_*}\Mc_N$ at the minimum. Differentiating the
relation $\gamma(t)^2 = \gamma(t)$ at $t=0$, we get
\[
P_*\gpp(0) + \gpp(0)P_* + 2X^2 = \gpp(0),
\]
from which we obtain the following two relations on the diagonal blocks of $\gpp(0)$
in the decomposition \eqref{eq:XP}:
\[
\label{eq:gpp}
\frac{1}{2} \oo{\gpp(0)} = -\oo{X^2} = -\ov{X}\vo{X},\qquad
\frac{1}{2} \vv{\gpp(0)} = \vv{X^2} = \vo{X}\ov{X}.
\]
Thus, since
$\vo{H_*} = \ov{H_*}^* = 0$ at the minimum, we have
\[
\begin{split}
  \croF{H_*,\gpp(0)} &= \Tr\prt{
    \begin{bmatrix}
      \oo{H_*} & 0 \\ 0 & \vv{H_*} \\
    \end{bmatrix}
    \begin{bmatrix}
      \oo{\gpp(0)} & \ov{\gpp(0)} \\ \vo{\gpp(0)} & \vv{\gpp(0)} \\
    \end{bmatrix}
  } \\
  &= 2\Tr\Big(-\oo{H_*}\ov{X}\vo{X}\Big)
  + 2\Tr\Big(\vv{H_*}\vo{X}\ov{X}\Big) \\
  &= 2\Tr\Big(-\vo{X}\oo{H_*}\ov{X}\Big)
  + 2\Tr\Big(\ov{X}\vv{H_*}\vo{X}\Big) \\
  &= \Tr\Big(X(\Op X)\Big),
\end{split}
\]
where the operator $\Op : T_{P_*}\Mc_N \to T_{P_*}\Mc_N$ is defined
as
\begin{equation}
  \begin{split}
  \Op X &\coloneqq
  P_{*} X (1-P_{*})H_* - H_{*} P_{*} X(1-P_{*}) + \hc \\
    &= \begin{bmatrix}
      0 & \ov{X}\vv{H_*} - \oo{H_*}\ov{X} \\
      \vv{H_*}\vo{X} - \vo{X}\oo{H_*} & 0 \\
    \end{bmatrix},
  \end{split}
\label{eq:O}
\end{equation}
where \enquote{$\hc$} stands for the transpose of
the previous expression.
Introducing the operator
\begin{equation}
\label{eq:hess_mani}
\Op + K_* :
\Tc_{P_*}\Mc_N \to  \Tc_{P_*}\Mc_N,
\end{equation}
one gets in the end
\[
E\prt{\gamma(t)} = E(P_*) + \frac{t^2}{2}\croF{X,
  (\Op + K_*) X} + o(t^2).
\]
At the critical point $P_{*}$, the second order expansion of
$E\prt{\gamma(t)}$ only depends on $X = \dot \gamma(0)$, a common
feature in constrained optimization. The operator $\Op + K_*$ can be
interpreted as the Hessian of the energy on the manifold, or
alternatively as the partial Hessian of the Lagrangian on $\Hc$. The
operator $\Op$ represents the influence of the curvature of the
manifold on the Hessian of $E$.

As $P_*$ is a nondegenerate minimum in the sense of Assumption~\ref{asmp2}, we have the
\begin{equation}
\label{eq:KO_spd}
\boxed{ \textbf{Second-order optimality condition: } \Forall
  X\in\Tc_{P_*}\Mc_N,\quad\croF{X, (\Op + K_*) X} \geqslant \eta \normF{X}^2.}
\end{equation}

\begin{remark}[Structure of $\Op$ and link with the \textit{Aufbau} principle]
  \label{rmk:O}
  Let $P_*$ be a nondegenerate minimizer of \eqref{mini_mani} in the
  sense of Assumption~\ref{asmp2}.
  Denoting by $A_{kl}$
  the component along $\phi_k\phi_l^*$ of the matrix $A\in\Hc$,
  the operator $\Op$
  defined in \eqref{eq:O} can alternatively be defined by
  \[
    \Forall X\in\Tc_{P_*}\Mc_N,\quad (\Op X)_{ia} = (\varepsilon_a -
    \varepsilon_i) X_{ia} \text{  and  } (\Op X)_{ai} = (\varepsilon_a -
    \varepsilon_i) X_{ai} \text{ for } i \in \mathcal{I}, \ a \notin \mathcal{I},
  \]
  where we have used the standard notation in chemistry of using the
  subscripts $i$ for occupied and $a$ for virtual orbitals
  ($\mathcal{I}$ is the set of occupied orbitals).

  In the case when $E(D) = \Tr(H D)$ for some fixed symmetric matrix
  $H \in \Hc$ (linear eigenvalue problem), then $K_{*} = 0$ and so
  \eqref{eq:KO_spd} is equivalent to the \emph{Aufbau} principle. This
  equivalence does not hold true in general for nonlinear models:
  \eqref{eq:KO_spd} is independent of the Aufbau principle, and $\eta$
  is unrelated to the gap
  $\nu = \min_{a \notin \mathcal{I}} \varepsilon_{a} - \max_{i \in \mathcal{I}}
  \varepsilon_{i}$ (equal to the lowest eigenvalue of the operator
  $\Op$). However in specific cases, such as the reduced Hartree-Fock
  or Gross-Pitaevskii model, where $K_{*} \ge 0$, we have
  $\eta \ge \nu$ and a positive gap is a sufficient (but not
  necessary) condition for optimality.
\end{remark}

\begin{remark}[Link with the Liouvillian]
 Another way to understand $\Op$ is to use the Liouvillian $\Lc_{H_*}$
  associated to $H_*$, which is defined by:
  \[
    \Forall A\in\Hc,\quad\Lc_{H_*}A \coloneqq [H_*,A].
  \]
  The action of $\Lc_{H_*}$ has a simple expression in the eigenvector decomposition
  $(\varepsilon_k, \phi_k)_{1\leqslant k \leqslant N_b}$ of $H_*$:
  \begin{equation}
    \label{eq:liouvil}
    \Forall 1 \lq k,l \lq N_b,\quad
    \Lc_{H_*}(\phi_k\phi_l^*) = (\varepsilon_k - \varepsilon_l)\phi_k\phi_l^*.
  \end{equation}
  Thus, we have
  \[  \Forall i \in \mathcal{I}, \ a \notin \mathcal{I}, \quad
   \Op(\phi_i\phi_a^* + \phi_a\phi_i^*)
    = (\varepsilon_a - \varepsilon_i)(\phi_i\phi_a^* + \phi_a\phi_i^*).
  \]
  Hence, one can easily check that,
  using again the decomposition \eqref{eq:XP},
  we have
  \begin{equation}
    \label{eq:Omega_Liouville}
    \Forall X\in\Tc_{P_*}\Mc_N,\quad\Op X = -[P_*,\Lc_{H_*} X] = -[P_*,[H_*,X]].
  \end{equation}
  This definition also provides a canonical way to extend the operator $\Op$, originally
  defined on $T_{P_*}\Mc_N$, to the whole space $\Hc$.
\end{remark}

\subsection{Fixed-point on a manifold}
\label{sec:cvg_mani}
Finally, we state a general abstract result that we will use to study the convergence
of optimization algorithms on manifolds.
\begin{lemma}
  \label{lem:jac}
  Let $\Mc$ be a smooth finite dimensional Riemannian
  manifold. Let $P_*\in\Mc$ and $f:U \to \Mc$ be a continuously
  differentiable mapping on a neighborhood $U$ of $P_*$ such that $f(P_*) = P_*$.
  Let $\text{\emph{d}} f(P_*): \Tc_{P_*}\Mc \to \Tc_{P_*}\Mc$ be the derivative
  of $f$ at $P_*$.
  If $\emph{d}f(P_*)$ verifies $r\prt{\text{\emph{d}} f(P_*)} < 1$ where
  $r\prt{\text{\emph{d}} f(P_*)}$ is the spectral radius of
  $\text{\emph{d}} f(P_*)$, then,
  for $P^0$ close enough to $P_*$, the fixed point iteration $P^{k+1} = f(P^k)$
  linearly converges to $P_*$ with asymptotic rate $r\prt{\text{\emph{d}} f(P_*)}$, in
  the sense that for all $\theta > 0$ there exists $C_\theta > 0$ such that, for all $P^0$ close enough to $P_*$,
  \[
    \norm{P^{k}-P_*} \leqslant C_\theta \prt{r(\text{\emph{d}} f(P_*))
      +\theta}^k \norm{P^{0}-P_*}.
  \]
\end{lemma}

\begin{proof}
  We use the
  notation presented in \cite[Chapter 1]{brownTopologyDifferentiableViewpoint1967}.
  Up to a restriction of $U$ to a smaller neighborhood of $P_*$,
  there exists a neighborhood $V$ of $0$ in $\Tc_{P_*}\Mc$ and $g: V \to U$ a
  diffeomorphic parametrization of the manifold such that $g(0) = P_*$ and
  $\d g(0) = \Id$ (take for instance the restriction to $V$ of the exponential map).
  Therefore, as $f$ is continuously differentiable, there exists a
  neighborhood $\widetilde V \subset V$ of~$0$ in $\Tc_{P_*}\Mc$ such that
  $F \coloneqq g^{-1} \circ f \circ g : \widetilde V \to V$ is a continuously differentiable
  map with fixed-point $0$ and $\d F(0) = \d f(P_*)$.
  As $r\prt{\d F(0)} = r\prt{\d f(P_*)} < 1$, we can find a neighborhood
  $V' \subset V$ of~$0$ in $T_{P_*}\Mc$ such that $F$ is a contraction in $V'$
  for some norm $\norm{\cdot}_\theta$, with contraction factor
  $r(\text{d} f(P_*)) +\theta$, $\theta > 0$ (see \cite{holmesFormulaSpectralRadius1968b} for more details).
  Therefore, we can apply the Banach fixed-point theorem to $F$  and we get
  that, for $x^0$ close enough to $0$, $x^{k+1} = F(x^k)$ converges to 0. Finally,
  for $P^0 = g(x^0)$,
  $P^{k+1} = g(x^{k+1}) = g(F(x^k)) = f(g(x^k)) = f(P^k)$ converges to
  $P_* = g(0)$, with asymptotic rate $r(\d f(P_*))$.
\end{proof}

\section{Algorithms and analysis of convergence}
\label{sec:algo}
\subsection{Direct minimization}
\label{sec:grad}
The gradient descent algorithm consists in following the steepest
descent direction with a fixed step $\beta$
at each iteration point.  As the iterations are constrained to stay on
the manifold, we have to
\begin{enumerate}
  \item project the gradient on the tangent space with $\Pi_{P^k}$
    to bring the steepest descent line $P^k~-~\beta\Pi_{P^k}\prt{\nabla E(P^k)}$
    back to the manifold at first order;
  \item retract the steepest descent line defined in the tangent space onto
    the manifold $\Mc_N$ by a nonlinear retraction $R$ mapping a neighborhood of $\Mc_N$ in $\Hc$ to $\Mc_N$.
\end{enumerate}
An example of retraction is given in Section~\ref{sec:retrac} and
we will assume that the retraction $R$ satisfies
\begin{asmp}
  \label{asmp3}
  $R:\Hc \to \Hc$ is of class $C^2$ and for all $P\in\Mc_N$ and $X\in\Hc$ small enough,
  \begin{equation*}
    R(P + X) \in\Mc_N \text{ and }
    R(P + X) = P + \Pi_P(X) + O(X^2).
  \end{equation*}
\end{asmp}
These two successive operations are sketched in Figure~\ref{fig:projo} and
the gradient descent algorithm is presented in Algorithm~\ref{algo:grad}.
\begin{figure}[h!]
  \centering
  \includegraphics[scale = 1]{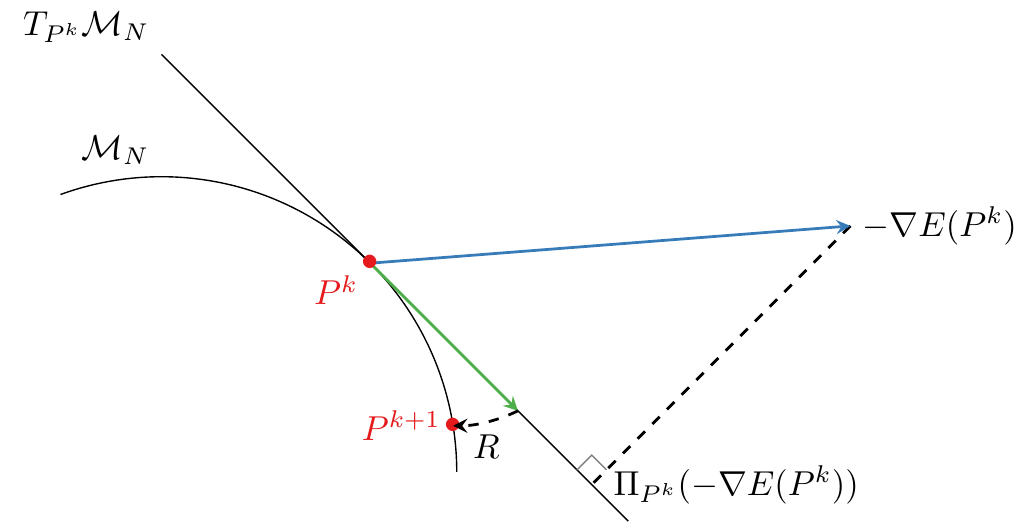}
  \caption[Retraction and projection for the gradient algorithm]{Projection on
    the tangent space for the gradient descent, and retraction to the
    manifold.}
  \label{fig:projo}
\end{figure}

\begin{algorithm}[H]
  \SetAlgoSkip{bigskip}
  \KwData{$P^0 \in \Mc_N$}
  \While{convergence not reached}{
    $P^{k+1} \coloneqq R\prt{P^k - \beta\Pi_{P^k}\prt{\nabla E(P^k)}}$ \;
  }
  \caption[Gradient descent]{Gradient descent}
  \label{algo:grad}
\end{algorithm}

At the continuous level, this algorithm can be seen as the
discretization of the flow $\dot{P} = - \Pi_P\nabla E(P)$. Note that,
by the use of the retraction $R$ and Assumption~\ref{asmp3}, the projection
step has no influence on the convergence of the algorithm for $\beta$
small. Indeed, by Assumption~\ref{asmp3},
\[
  \begin{split}
    \Forall P\in \Mc_N,\quad R\prt{P-\beta \Pi_P\prt{\nabla E(P)}} &=
    P -\beta\Pi_P\prt{\Pi_P\prt{\nabla E(P)}} + O(\beta^2) \\ &=
    P -\beta\Pi_P\prt{\nabla E(P)} + O(\beta^2)
  \end{split}
\]
and thus the first order expansion is the same with or without the
projection step. The reason we use this projection step is that it is
convenient to interpret $\Pi_{P^{k}} \nabla E(P^{k})$ as a residual.

The following theorems state that, for $\beta$ small enough,
Algorithm~\ref{algo:grad} globally converges in the sense that
$\Pi_{P^k}\nabla E(P^k) \to 0$ and locally converges in the sense that
$P^k \to P_*$ if $P^0$ is close enough to $P_*$.
\begin{theorem}
  \label{thm:grad_glob}
  Let $E:\Hc\to\R$ satisfy Assumption~\ref{asmp1}
  and $R:\Hc \to \Hc$ satisfy Assumption~\ref{asmp3}. There exists $\beta_0 > 0$ such that for all $0 < \beta \le \beta_0$ and all
  $P^0\in\Mc_N$, the iterations
  \[
    P^{k+1}\coloneqq R\prt{P^k - \beta \Pi_{P^k}\prt{\nabla E(P^k)}}
  \]
  satisfy the following properties:
  \begin{enumerate}
    \item $(E(P^k))_{k \in \N}$ is a nonincreasing sequence converging to some
      critical value $E_{\rm c}$ of $E$ on $\Mc_N$;
    \item when $k$ goes to infinity, $\Pi_{P^k}\nabla E(P^k) \to 0$,
  $\|P^{k+1}-P^k\|_{\rm F} \to 0$ and $d(P^k,A_{\rm c}) \to 0$ where
  $A_{\rm c}$ is one of the connected components of
  $C(E_{\rm c})\coloneqq\{ P \in \Mc_N \, | \, E(P)=E_{\rm c} \; \text{and} \;  \Pi_{P}\prt{\nabla E(P)}=0\}$
  .
\end{enumerate}
\end{theorem}
\begin{proof}
  As $E:\Hc \to \R$ and $R:\Hc \to \Hc$ are $C^2$, and $\Mc_N$ is compact, we can use the expansion of Assumption~\ref{asmp3} and obtain that there exists a constant $C \in \R_+$ such that for all $0 \le \beta \le 1$,
  \begin{align*}
\Forall k \in \N, \quad    E(P^{k+1})  \le  E(P^k) - \beta\normF{\Pi_{P^k}\nabla E(P^k)}^2 + C \beta^2 \normF{\Pi_{P^k}\nabla E(P^k)}^2
  \end{align*}
  Therefore, we have for $\beta > 0$ small enough,
  \[
   \Forall k \in \N, \quad    E(P^{k+1}) \leqslant E(P^k) -\frac{\beta}{2}\normF{\Pi_{P^k}\nabla E(P^k)}^2.
  \]
  This shows that the sequence $(E(P^k))_{k\in\N}$ is nonincreasing.  As $E$
  is continuous on the compact set $\Mc_N$,
  $(E(P^k))_{k\in\N}$ is bounded and hence converges to some $E_{\rm c} \in \R$.
  Moreover,
  $$
  \sum_{k \in \N} \normF{\Pi_{P^k}\nabla E(P^k)}^2 < \infty,
  $$
  which implies that $\Pi_{P^k}\nabla E(P^k) \to 0$ when
  $k \to \infty$. It follows that $\|P^{k+1}-P^k\|_{\rm F} \to 0$ when
  $k \to \infty$.

  Let $B$ be the non-empty compact set of accumulation points of
  $(P^{k})_{k \in \N}$. By continuity of $E$ and
  $P \mapsto \Pi_{P}\nabla E(P)$, it follows that $B \subset C(E_{\rm c})$. Assuming that
  $d(P^{k}, B)$ does not go to zero, we can extract a subsequence at
  finite distance of $B$ which converges to a point in $B$, a
  contradiction. Assume that $B$ is disconnected: it is then the union of
  two compact subsets $B_{1}$ and $B_{2}$ at positive distance from each other.
  Since $P^{k+1} - P^{k} \to 0$, there is an infinite number of points
  in $(P^{k})_{k \in \N}$ at distance greater or equal to $\eta > 0$ from both $B_{1}$ and
  $B_{2}$, from which we can extract a point in $B$, a contradiction.
  It follows that $B$ is connected, hence the result.
\end{proof}
This result implies in particular the convergence of the sequence
$(P^{k})_{k \in \N}$ in the generic case where critical points are
isolated. If this is not the case but $E$ and $R$ are analytic,
convergence can be shown following the approach
in~\cite{levittConvergenceGradientbasedAlgorithms2012b} based on \L
ojasiewicz inequality.

\begin{theorem}
  \label{thm:grad}
  Let $E: \Hc \to \R$ satisfy Assumption~\ref{asmp1} and Assumption~\ref{asmp2} with $P_*$
  a local minimizer of \eqref{mini_mani}.
  Let $R:\Hc \to \Hc$ satisfy Assumption~\ref{asmp3}.
  Then, if $P^0\in\Mc_N$ is close enough to $P_*$, the iterations
  \[
    P^{k+1} \coloneqq R\prt{P^k - \beta \Pi_{P^k}\prt{\nabla E(P^k)}}
  \]
  linearly converge to $P_*$ for $\beta>0$ small enough, with asymptotic rate
  $r(1 - \beta J_\text{\emph{grad}})$ where $J_\text{\emph{grad}} \coloneqq \Op + K_*$.
\end{theorem}
\begin{proof}
  In order to prove convergence, one can apply Lemma~\ref{lem:jac}
  to the function $f: \Mc_N\to\Mc_N$ defined by
  \[
    f(P) \coloneqq R(P - \beta \Pi_P\prt{\nabla E(P)}),
    \label{eq:f}
  \]
  for which we know by the first order optimality condition
  that $P_*$ is a fixed-point.

  We compute explicitly $\d f(P_*)$ using the
  second-order optimality condition \eqref{eq:KO_spd}. To this end, take
  $X\in\Tc_{P_*}\Mc_N$ and a smooth path
  $\gamma : I\to\Mc_N$ defined on a real interval $I$ containing 0
  such that $\gamma(0) = P_*$ and $\dot\gamma(0) = X$.
  We want to expand to the first order in $t$ the following expression:
  \[
    f(\gamma(t)) = R\prt{\gamma(t) - \beta \Pi_{\gamma(t)}
      \prt{\nabla E(\gamma(t))}}.
  \]
  First, we focus on the projection of $H(\gamma(t))$
  on $T_{\gamma(t)}\Mc_N$:
  \[
    \begin{split}
      \Pi_{\gamma(t)}H(\gamma(t)) &= \gamma(t)H(\gamma(t))(1-\gamma(t)) +
      \hc \\
      &= (P_* + tX)\prt{H_* + t\prt{\nabla^2E(P_*) X}}(1 - P_* - tX) + \hc +
      O(t^2) \\
      &= t\big[ P_*\prt{\nabla^2E(P_*) X}(1-P_*) + \hc \big]
    + t\big[ XH_*(1-P_*) - P_*H_*X + \hc \big] +
      O(t^2) \\
      &= t(K_* + \Op)X + O(t^2).
    \end{split}
  \]
  Inserting this into the expansion of $f(\gamma(t))$, using Assumption~\ref{asmp3}
  and the fact that $\Pi_{\gamma(t)}X~=~\Pi_{P_*}X~+~O(t^2)$, gives
  \[
    f(\gamma(t)) = R\big(\gamma(t) - \beta t
    (\Op + K_*) X + O(t^2)\big) = P_* + t\big(X - \beta (\Op + K_*) X\big) + O(t^2).
  \]
  Therefore,
  \[
    \d f(P_*) X = \big(1 - \beta (\Op + K_*)\big) X.
  \]
  As the second-order optimality condition \eqref{eq:KO_spd}
  shows that $\Op + K_*$ is positive definite
  on $\Tc_{P_*}\Mc_N$, for $\beta$ small enough, the spectral radius $r(\d f(P_*))$
  of the derivative $\d f(P_*)$, is less than 1, which concludes the proof.
\end{proof}

\subsection{Damped self-consistent field}
\label{sec:scf}
The damped SCF algorithm is a damped version of the Roothaan algorithm
\cite{cancesConvergenceSCFAlgorithms2000c,levittConvergenceGradientbasedAlgorithms2012b}
and is presented in Algorithm~\ref{algo:scf}, under the assumption that the strong
\emph{Aufbau} principle is satisfied, and represented in Figure~\ref{fig:scf}.
Note that it is well defined only if
$\varepsilon_N^k < \varepsilon_{N+1}^k$ for all $k\in\N$. We introduce the
nonlinear operators:
\begin{enumerate}
  \item $A(H) \coloneqq \mathbf{1}_{(-\infty,\varepsilon_N(H)]}(H)$, with
    $\varepsilon_N(H)$ the lowest $N^{\rm th}$ eigenvalue of $H$ and where we
    recall that $\displaystyle{
      \mathbf{1}_{(-\infty,\mu]}(H) \coloneqq
      \sum_{\varepsilon_i \leq \mu} \phi_i\phi_i^*
    }$, the $\phi_{i}$'s being orthonormal eigenvectors of $H$
    associated to the eigenvalues $\varepsilon_i$;
  \item $\auf(P) = A(H(P))$ or, equivalently,
    $\displaystyle{\auf(P) \coloneqq \sum_{i=1}^N \phi_i \phi_i^*}$ where the
    $\phi_i$'s are orthonormal eigenvectors associated to the lowest $N$
    eigenvalues of $H(P)$.
\end{enumerate}

\begin{algorithm}[H]
  \SetAlgoSkip{bigskip}
  \KwData{$P^0 \in \Mc_N$}
  \While{convergence not reached}{
    solve $\displaystyle{\begin{cases}
        H(P^k)\phi_i^k = \varepsilon_i^k \phi_i^k,\ \varepsilon_1^k \leqslant \cdots \leqslant
        \varepsilon_N^k < \varepsilon_{N+1}^k
        \leq \cdots \leq \varepsilon^k_{N_b}\\
        \cro{\phi_i^k, \phi_j^k} = \delta_{ij}, \\
      \end{cases}}$\;
    $\auf(P^k) \coloneqq \displaystyle{\sum_{i=1}^N \phi^{k}_i\prt{\phi^{k}_i}^*}$\;
    $P^{k+1} \coloneqq R\prt{P^k + \beta\Pi_{P^k}\prt{\auf(P^k) - P^k}}$\;
  }
  \caption[Damped SCF algorithm]{Damped SCF algorithm}
  \label{algo:scf}
\end{algorithm}

\begin{figure}[h!]
  \centering
  \includegraphics[scale = 1]{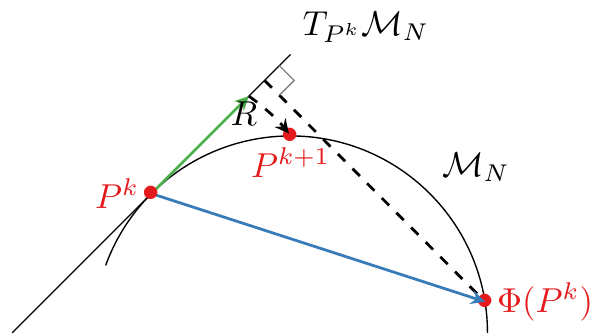}
  \caption{Retraction for the damped SCF algorithm.}
  \label{fig:scf}
\end{figure}

The following theorem states that, under the condition that there
is a gap between the smallest $N^{\rm th}$ and $(N+1)^{\rm st}$ eigenvalues of the
Hamiltonian $H_*$, Algorithm~\ref{algo:scf} locally converges for $\beta$ small
enough.
\begin{theorem}
  \label{thm:scf}
  Let $E: \Hc \to \R$ and $P_* \in \Mc_N$ satisfy Assumption~\ref{asmp1} and Assumption~\ref{asmp2} and $R~:~\Hc~\to~\Hc$ satisfy Assumption~\ref{asmp3}.  Assume that $P_*$ satisfies the strong {\em Aufbau} principle
  \[
    \auf(P_*) = P_* \text{ and }
    \nu \coloneqq \varepsilon_{N+1} - \varepsilon_N > 0,
  \]
  where $(\varepsilon_i)_{1\leq i \leq N_b}$ are the eigenvalues of $H_*$
  ranked in nondecreasing order.

  Then, for $\beta>0$ small enough and $P^0\in\Mc_N$ close enough to $P_*$, the iterations
  \[
    \label{eq:scf_ite_FD}
    P^{k+1} \coloneqq R\prt{P^k + \beta \Pi_{P^k}\prt{\auf(P^k) - P^k}}
  \]
  are well-defined and $P^k$ linearly converges to $P_*$, with asymptotic rate
  $r(1 - \beta J_\text{\emph{SCF}})$ where $J_\text{\emph{SCF}} \coloneqq~1+~\Op^{-1}K_*$.
\end{theorem}

\begin{proof}
  In order to prove convergence, we apply Lemma~\ref{lem:jac}
  to the function $f : \Mc_N \to \Mc_N$ defined by
  \[
    f(P) \coloneqq R(P + \beta \Pi_P\prt{\auf(P)-P}),
  \label{eq:f_scf}
  \]
  for which $P_*$ is a fixed-point.

  First, we compute the derivative of $\auf = A\circ H$ at the minimizer $P_*$
  to get
  \[
    \d \auf(P_*) = \d {A}(H_*)\nabla^2 E(P_*).
  \]

  Now, to compute $\d {A}(H_*)$, note that, as there is a gap
  $\varepsilon_{N+1} > \varepsilon_N$ at the minimum,
  we can find a contour $\Cc$  in the complex plane enclosing the lowest $N$
  eigenvalues of $H_*$ (Figure~\ref{fig:aufbau}) such that
  \begin{equation}
    \label{eq:At}
    A(H_*) = \frac{1}{2\pi\i}\oint_\Cc \frac{1}{z-H_*}\ \d z
  \end{equation}
  (see \cite{highamFunctionsMatrices2008,katoPerturbationTheoryLinear1995}
  for more details on spectral calculus, contour integrals and
  perturbation theory for functions of matrices).
  By continuity, we also have
  \[
    A(H) = \frac{1}{2\pi\i}\oint_\Cc \frac{1}{z-H}\ \d z
  \]
  for $H$ in a neighborhood of $H_*$.
  \begin{figure}[h!]
    \centering
    \includegraphics[scale = 1]{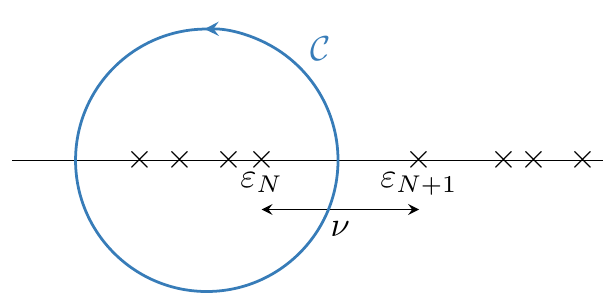}
    \caption[Definition of $A$]{Definition of $A$ and graphical interpretation
      of the \emph{Aufbau} principle and the existence of a gap.}
    \label{fig:aufbau}
  \end{figure}
  Then, one can use the expression \eqref{eq:At} of $A$
  and the expansion for $H$ in a neighborhood of $H_*$
  \[
    \Forall z \in\Cc,\quad \frac{1}{z-H} = \frac{1}{z-H_*}(H-H^*)\frac{1}{z-H_*}
    + O(\norm{H-H_*}_F^2)
  \]
  to get
    \[
    \begin{split}
      \Forall h\in\Hc,\quad
      \d {A}(H_*) h &= \frac{1}{2\pi\i}\oint_\Cc \frac{1}{z-H_*} h
      \frac{1}{z-H_*}\ \d z \\
      &= \sum_{k=1}^{N_b} \sum_{l=1}^{N_b}
      \prt{\frac{1}{2\pi\i}\oint_\Cc \frac{1}{z-\varepsilon_k} h_{kl}
      \frac{1}{z-\varepsilon_l}\ \d z} \phi_k\phi_l^*,\\
    \end{split}
  \]
  where $h_{kl}=\phi_k^*h\phi_l$.
  Now, let us denote by $1\leqslant i \leqslant N$
  the occupied orbitals ($\varepsilon_i$ is inside $\Cc$)
  and by $N+1 \leqslant a \leqslant N_b$ the virtual ones ($\varepsilon_a$ is outside $\Cc$).
  Then,
  \[
    \oint_\Cc \frac{1}{z-\varepsilon_i}\frac{1}{z-\varepsilon_a}\ \d z =
    \begin{cases}
      \displaystyle{\frac{1}{\varepsilon_i - \varepsilon_a} \quad\text{if
          $1\leqslant i \leqslant N < a \leqslant N_b $};} \\
      0 \quad\text{otherwise.}
    \end{cases}
  \]
  Thus, the sum becomes
  \[
    \d {A}(H_*) h
    = \sum_{i=1}^{N}\sum_{a=N+1}^{N_b} \frac{1}{\varepsilon_i-\varepsilon_a}
    \Big( h_{ia}\phi_i\phi_a^* + h_{ai}\phi_a\phi_i^*\Big)
    = -\Omega_*^{-1}\Pi_{P_*}h,
  \]
   and we finally get
  \[
    \Forall X\in\Tc_{P_*}\Mc_N,\quad \d \auf(P_*) X = -\Op^{-1}K_* X.
  \]

  Now, we compute the derivative of $f$ at point $P_*$.
  Let $X\in\Tc_{P_*}\Mc_N$ and $\gamma: I \to \Mc_N$ a smooth path defined on a
  real interval $I$ containing $0$ such that
  $\gamma(0) = P_*$ and $\dot{\gamma}(0) = X$.  First, we expand $\gamma(t)$
  around $0$ and $\auf$ around $\gamma(0) = P_*$ to obtain
  \[
    f(\gamma(t)) = P_* + t\Pi_{P_*}\prt{(1-\beta) + \beta \d \auf(P_*)} X + O\prt{t^2}.
  \]
  Thus, for $X\in\Tc_{P_*}\Mc_N$,
  \[
    \d f(P_*) X = \prt{(1-\beta) - \beta
      \Op^{-1}K_*} X.
  \]

  To conclude this proof, we compute the spectral radius of
  \[
    \d f(P_*) = (1-\beta) - \beta \Op^{-1}K_* = 1 - \beta \prt{1 + \Op^{-1}K_*}.
  \]
  First, notice that
  \[
    1 + \Op^{-1}K_* = 1 + \Op^{-1/2}\prt{\Op^{-1/2}K_*\Op^{-1/2}}\Op^{1/2}
  \]
  and thus, $1 + \Op^{-1}K_*$ and the symmetric operator
  $1 + \Op^{-1/2}K_*\Op^{-1/2}$ have
  the same eigenvalues.  Moreover, using the second-order optimality condition
  \eqref{eq:KO_spd}, with $ X = \Op^{-1/2} Y$, we get
  \begin{equation}
    \label{eq:jac_scf}
    \begin{split}
      \Forall Y\in T_{P_*}\Mc_N,\quad
      \croF{Y,\prt{1 + \Op^{-1/2}K_*\Op^{-1/2}} Y}
      &\geqslant \eta \croF{Y,\Op^{-1} Y} \\
      &\geqslant \frac{\eta}{\norm{\Op}_{\rm op}}\normF{Y}^2,
    \end{split}
  \end{equation}
  with $\norm{\Op}_{\rm op}$ the operator norm associated to $\normF{\cdot}$.
  Thus, all the eigenvalues of $ 1 + \Op^{-1/2}K_*\Op^{-1/2}$, hence of
  $1 + \Op^{-1}K$, are real and positive.
  Consequently,
  for $\beta$ small enough, the spectral radius
  $r\prt{\d f(P_*)}$ is less than 1 and we conclude by applying Lemma~\ref{lem:jac}.
\end{proof}

\begin{remark}[Case when the \emph{Aufbau} principle is not satisfied]
  \label{rmk:no_aufbau}
  In the case when the minimizer $P_*$ does not verify the \emph{Aufbau}
  principle, but does satisfy the condition that the eigenvalues of $1
  + \Op^{-1}K_*$ are positive
  (note that $\Op$ is not positive when the \emph{Aufbau} principle is not
  verified, but $1 + \Op^{-1}K_*$ might still have only positive eigenvalues),
  the damped SCF still converges locally to $P_*$ for $\beta>0$ small enough
  if we change the way we select the occupied orbitals to build $\auf(P)$
  (in this case, we do not pick those
  associated to the smallest $N$ eigenvalues of $H(P)$, but those corresponding to the occupied orbitals of $P_*$).
\end{remark}

We conclude this section by proving the local convergence of the
non-retracted variant of Algorithm~\ref{algo:scf}.

\begin{algorithm}[H]
  \SetAlgoSkip{bigskip}
  \KwData{$P^0 \in \Mc_N$}
  \While{convergence not reached}{
    solve $\displaystyle{\begin{cases}
        H(P^k)\phi_i^k = \varepsilon_i^k \phi_i^k,\ \varepsilon_1^k \leqslant \cdots \leqslant
        \varepsilon_N^k < \varepsilon_{N+1}^k
        \lq \cdots \lq \varepsilon^k_{N_b}\\
        \cro{\phi_i^k, \phi_j^k} = \delta_{ij}, \\
      \end{cases}}$\;
    $\auf(P^k) \coloneqq \displaystyle{\sum_{i=1}^N \phi^{k}_i\prt{\phi^{k}_i}^*}$\;
    $P^{k+1} \coloneqq {P^k + \beta\Pi_{P^k}\prt{\auf(P^k) - P^k}}$\;
  }
  \caption[Non-retracted damped SCF algorithm]{Non-retracted damped SCF algorithm}
  \label{algo:nonretracted_scf}
\end{algorithm}

\begin{theorem}
  \label{thm:linear_mixing}
  Let $E: \Hc \to \R$ and $P_*$ satisfy Assumption~\ref{asmp1} and Assumption~\ref{asmp2}.  Moreover, assume that
  \[
    \auf(P_*) = P_* \text{ and }
    \nu \coloneqq \varepsilon_{N+1} - \varepsilon_N > 0 \text{ (strong
      {\em Aufbau} principle)},
  \]
  where $(\varepsilon_i)_{1\lq i \lq N_b}$ are the eigenvalues of $H_*$
  ranked in nondecreasing order.

  Then, for $\beta>0$ small enough and $P^0\in\Hc$ close enough to $P_*$ and with trace $N$, the iterations
  \begin{equation}\label{eq:density_mixing}
    P^{k+1} \coloneqq P^k + \beta \prt{\auf(P^k) - P^k}
  \end{equation}
  are well-defined and $P^k$ linearly converges to $P_*\in\Mc_N$, with asymptotic rate
  $\max(r(1 - \beta J_\text{\emph{SCF}}), 1-\beta)$ where $J_\text{\emph{SCF}} \coloneqq 1 + \Op^{-1}K_*$.
\end{theorem}

Note that the iterates $P^k$ defined by \eqref{eq:density_mixing} have trace $N$ but do not lay on the manifold $\Mc_N$ in general.

\begin{proof}
  The proof follows that of Theorem~\ref{thm:scf}. This time, we need to
  compute the Jacobian matrix of
  $f : \Hc \ni P \mapsto P + \beta\prt{\auf(P) - P} \in \Hc$ at the minimizer
  $P_*\in\Mc_N$.  As we work in the whole space $\Hc$, the Jacobian
  matrix has the form, in the
  decomposition $\Hc = T_{P_*}\Mc_N \oplus (T_{P_*}\Mc_N)^\perp$,
  \[
    \d f(P_*) = \begin{bmatrix}
      1 - \beta J_\text{SCF} & \times \\
      0 & 1 - \beta \\
    \end{bmatrix},
  \]
  where $J_\text{SCF} = 1 + \Op^{-1} K_*$ has been computed in the proof of Theorem~\ref{thm:scf}.
  Hence, this time the algorithm converges to $P_*\in\Mc_N$ as long as $\beta$ is such that
  $\max(r(1-\beta J_\text{SCF}), 1-\beta) < 1$.
\end{proof}

In LDA and GGA Kohn-Sham models~\cite{martin2004}, the mean-field Hamiltonian $H(P)$ is actually a function $\widetilde H(\rho_P)$ of the density $\rho_P$ associated with the density matrix $P$. Since the map $P \mapsto \rho_P$ is linear, \eqref{eq:density_mixing} can be rewritten as
\[
  \rho^{k+1} = (1-\beta)\rho^k + \beta \Psi(\rho^k),
\]
where $\Psi(\rho)=\rho_{A(\widetilde H(\rho))}$. We can therefore
interpret \eqref{eq:density_mixing} as the equivalent density matrix
formulation of this density mixing algorithm.

\subsection{Comparison}
\label{sec:com}
In this section, we proved the local convergence of Algorithm~\ref{algo:grad} and Algorithm~\ref{algo:scf}.
and we obtained asymptotic convergence rates. On the tangent space, both
Jacobian matrices are of the form $1 - \beta J$ where $J$ has positive real spectrum and
\begin{itemize}
  \item for the gradient descent: $J_\text{grad} = K_* + \Op$, which is self-adjoint for
    the Frobenius inner product;
  \item for the damped SCF algorithm if the strong {\em Aufbau}
    principle is satisfied at $P_*$: $J_\text{SCF} = 1+~\Op^{-1} K_*$,
    which is self-adjoint for the inner product
    $\cro{\cdot,\cdot}_\Op \coloneqq \croF{\Op\cdot,\cdot}$.
\end{itemize}
One can notice that, \emph{in the linear regime}, the SCF iterations
correspond to a matrix splitting of the gradient iterations.
Whether this results in a faster method or not depends not only on the
relative conditioning of the iteration matrices but also on the
relative cost of each step.

To have the fastest convergence, we want the eigenvalues of $1 - \beta J$ to
be as close to 0 as possible. If we denote by $\lambda_1$ (resp. $\lambda_N$)
the smallest (resp. largest) eigenvalue of $J$, the optimal step $\beta_*$ is
the minimizer of $\min_{\beta}\max\set{|1-\beta\lambda_1|, |1-\beta\lambda_N|}$,
which is given by
\[
  \beta_* = \frac{2}{\lambda_1 + \lambda_N}.
\]
Then, the rate of convergence  is, with $\kappa \coloneqq \lambda_N/\lambda_1$ the spectral condition number of $J$,
\[
  r = \frac{\kappa-1}{\kappa+1}.
\]
Now, we can evaluate the conditioning of $J$ for the two algorithms :
\begin{itemize}
  \item for the gradient descent, we have
    \begin{equation}
    \label{eq:bound_grad}
      \kappa(J_\text{grad}) \leqslant \frac{\norm{\Op}_{\rm op} + \norm{K_*}_{\rm op}}{\eta},
    \end{equation}
    where $\eta$ is the coercivity constant in the nondegeneracy
    Assumption~\ref{asmp2}. First, the smaller $\eta$, the more difficult the
    convergence. Note however that there is no relationship in general
    between $\eta$ and the gap $\nu$. Second, the bigger
    $\norm{\Op}_{\rm op} = \varepsilon_{N_b} - \varepsilon_1$, the
    more difficult the convergence. In particular, for models arising
    from the discretization of partial differential equations,
    $\varepsilon_{N_b}~-~\varepsilon_1~\to~\infty$ when the
    discretization is refined. In practice, this issue is solved by
    preconditioning (see Remark~\ref{rmk:prec}).
  \item for the damped SCF algorithm, a naive bound would be
    \[
      \kappa(J_\text{SCF}) \leqslant \norm{\Op}_{\rm op} \frac{1 +
        \nu^{-1}\norm{K_*}_{\rm op}}{\eta}.
    \]
    In this bound the right-hand side diverges when $\norm{\Op}_{\rm
      op} \to \infty$ as above, whereas the left-hand side may actually remain bounded. For instance, under the uniform coercivity assumption \cite{Cances2012}
    \[
      \Forall X\in T_{P_*}\Mc_N,\quad
      \croF{X,\prt{\Op + K_*} X}
      \geqslant \widetilde\eta\croF{\Op X, X},
    \]
    with $\widetilde{\eta}$ independent of $N_b$ (which is often the case in practice),
     we have
    \[
      \kappa(J_\text{SCF}) \leqslant \frac{1 + \nu^{-1}\norm{K_*}_{\rm
        op}}{\widetilde\eta}.
    \]
    In contrast with the bound \eqref{eq:bound_grad},  we can see that the smaller the gap $\nu$, the slower the convergence.
\end{itemize}

As a special case, if we consider the case where the Hessian
$\nabla^2E \equiv 0$, \ie a linear eigenvalue problem. Then the SCF
algorithm converges in one iteration, which is consistent with
$J_\text{SCF} = 1$. The gradient descent with optimal step locally
converges with asymptotic rate $r = \frac{\kappa - 1}{\kappa + 1}$
where $\kappa = \frac{\varepsilon_{N_{b}} -
  \varepsilon_{1}}{\varepsilon_{N+1} - \varepsilon_{N}}$.

The convergence rates we derived in Theorem~\ref{thm:grad} and
Theorem~\ref{thm:scf} are consistent with well-known convergence issues,
for instance the failure of the simple damped SCF algorithm to
converge for systems with small gaps \cite{rudberg2012difficulties} (although
Section~\ref{sec:DFTK} shows this is not necessarily true for more sophisticated acceleration methods).

\begin{remark}[Preconditioning]
  \label{rmk:prec} We discuss here the extension of Theorem~\ref{thm:grad} to the
  preconditioned gradient descent:
  \[
  P^{k+1} \coloneqq R\prt{P^k - \beta \Pi_{P^k }B\Pi_{P^k}\prt{\nabla E(P^k)}}
  \]
  with $B: \mathcal H \to \mathcal H$ a symmetric positive definite
  preconditioner. If we denote by
  $\widetilde{B}_* \coloneqq \Pi_{P_*}B\Pi_{P_*}$ its restriction to
  the tangent plane, the Jacobian matrix of the gradient becomes
  $1 - \beta \widetilde{B}_*(\Op + K_*)$ where $(\Op + K_*)$ is
  positive definite (under Assumption~\ref{asmp2}) and the proof of
  local convergence for $\beta$ small enough follows exactly in the
  same way, using the positive definiteness of $\widetilde{B}_{*}$ to
  show that $\widetilde B_{*}(\Op + K_*)$ has real positive spectrum.
  The same analysis holds true for the preconditioned SCF algorithm.
  In practice, preconditioning is a crucial tool to accelerate
  iterations, in particular in order to achieve mesh- and domain-size
  independence of the number of iterations for discretized partial
  differential equations. However, we are interested here in the
  intrinsic aspects of each algorithm (direct minimization \emph{vs}
  SCF) and the influence of physical parameters (e.g. the gap $\nu$),
  so that the study of preconditioned algorithms is not in the scope
  of this paper.
\end{remark}

\begin{remark}[Dielectric operator]
  In the context of Kohn-Sham density functional theory, the operator
  $\prt{ 1 + K_*\Op^{-1}}^{-1}$, the transpose of the inverse of the
  Jacobian of the simple SCF mapping, is known as the dielectric
  operator: it represents the infinitesimal change in the
  self-consistent Hamiltonian $H(P_{*})$ in response to a change in
  the energy functional. Our results show that this operator is
  well-defined and has real positive spectrum, with no assumption on
  the sign of Hartree-exchange-correlation kernel $K_{*}$, recovering
  in an algebraic framework the results of
  \cite{dederichs1983self,gonze1996towards} obtained using a different
  variational principle.
\end{remark}

\section{Numerical tests}
\label{sec:num_res}

We present here some numerical experiments to illustrate our theoretical results,
explore their limits and investigate the global behavior of the
algorithms.
First, we start by specifying the retraction $R$ that we use in our numerical
tests.
In Section~\ref{sec:toy_model}, we use a simple toy model for which we can
control the gap and analytically compute the exact minimizer: this allows us
to study the impact of the gap on the convergence of Algorithm~\ref{algo:grad} and Algorithm~\ref{algo:scf}.
In Section~\ref{sec:chaos}, we show that simple (nondamped) SCF iterations can
exhibit chaotic behavior for some nonlinearities.
Then, in Section~\ref{sec:validate}, we report numerical tests for a 1D
Gross-Pitaevskii model ($N=1$) and its fermionic version for $N=2$.
Finally, in Section~\ref{sec:DFTK}, we present results obtained with a more
realistic case: Silicon in the framework of the Kohn-Sham DFT.

\subsection{The retraction $R$}
\label{sec:retrac}

We choose the following algorithm: for a given symmetric matrix $\widetilde P$ close to $\mathcal M_{N}$ with
eigendecomposition $\widetilde P = V \widetilde D V^{*}$ with
$\widetilde D$ diagonal and $V$ orthogonal, we set the diagonal matrix $D$ as
\begin{align*}
  D_{ii} =
  \begin{cases}
    1& \text{ if $\widetilde D_{ii} > 0.5$}\\
    0& \text{ otherwise}\\
  \end{cases}.
\end{align*}
and $R(\widetilde P) = V D V^{*}$. When $\widetilde P$ is close to
$\mathcal M_{N}$, its eigenvalues are close to either $0$ or $1$.
Given a contour $\Cc$ enclosing only the eigenvalues close to 1
, $R$ has the following explicit
expression
\begin{equation}
  \label{eq:C_retrac}
  R(P) = \frac{1}{2\pi\i} \oint_\Cc \frac{1}{z-P}\ \d z.
\end{equation}
Therefore, it follows from arguments similar to those used in the proof of
Theorem~\ref{thm:scf} that $R$ is analytic and satisfies
Assumption~\ref{asmp3}.

\subsection{A toy model with tunable spectral gap}
\label{sec:toy_model}
We work here in the very simple framework of real density matrices of order 2,
\ie the $2\times2$ real matrices $P$ such that $P^* = P$, $P^2=P$ and
$\Tr(P)=1$. Then, we consider the following minimization problem, with a parameter $\varepsilon \gq 0$:
energy functional
\[
E_\varepsilon(P) \coloneqq \Tr\prt{\prt{P -
    \begin{bmatrix}1 & \varepsilon \\ \varepsilon & 0\end{bmatrix}
  }^2},
\]
for parameters $\varepsilon \gq 0$. The gradient and Hessian of $E$ are
\begin{align*}
  H_{\varepsilon}(P) &= 2\prt{P -
    \begin{bmatrix}1 & \varepsilon \\ \varepsilon & 0\end{bmatrix}},\\
  \nabla^{2} E_{\varepsilon}(P)&= 2.
\end{align*}
Simple computations show that the set of rank-$1$ projectors on $\R^2$ can be parameterized as
\[
  \Mc_1 \coloneqq \set{P(a,b) =
    \begin{bmatrix}1-a & b \\ b & a\end{bmatrix} \;\middle|\; a\in[0,1],\ b =\pm\sqrt{a(1-a)
  }}.
\]
The eigenvalues of $H_\varepsilon$ at $P(a,b)\in\Mc_1$ are
$\pm 2\sqrt{a^2 + (b-\varepsilon)^2}$. The gap is thus
$\nu(a,b)\coloneqq4\sqrt{a^2 + (b-\varepsilon)^2}$.

\subsubsection*{The case $\varepsilon=0$}Here, the unique minimum is clearly
\[
  P(0,0) = \begin{bmatrix} 1 & 0 \\ 0 & 0 \end{bmatrix} \in \Mc_1
\]
and the gap is zero. Since $\nabla^{2} E = 2$, this minimum satisfies
Assumption~\ref{asmp2} with $\eta=2$.

\subsubsection*{The case $\varepsilon > 0$} We compute
\[
  E_\varepsilon(P(a,b)) = 2(a+\varepsilon^2 - 2\varepsilon b),
\]
and therefore
\[
  E_\varepsilon\prt{P({a,\sqrt{a(1-a)}})} \leqslant E_\varepsilon\prt{P(a,-\sqrt{a(1-a)})}.
\]
Hence, to compute the minimizer of the energy, we can restrict ourselves to
the one-dimensional manifold
\[
P(a) =
\begin{bmatrix}
  1-a & \sqrt{a-a^2} \\ \sqrt{a-a^2} & a
\end{bmatrix}
\]
with $a \in [0,1]$.
Then, the energy is
\begin{equation}
  \label{eq:nrj_toy_a}
  E_\varepsilon(P(a)) = 2\prt{a+\varepsilon^2 - 2\varepsilon \sqrt{a(a-1)}}.
\end{equation}
The first-order condition yields
\begin{align*}
  a = \frac{1 \pm \sqrt{1 - \frac{4\varepsilon^2}{1+4\varepsilon^2}}}{2},
\end{align*}
with the lowest energy achieved at
\begin{align*}
  a(\varepsilon) \coloneqq \frac{1 - \sqrt{1 - \frac{4\varepsilon^2}{1+4\varepsilon^2}}}{2}.
\end{align*}
The gap
$\nu(\varepsilon) \coloneqq 4\sqrt{a(\varepsilon)^2 +
  \prt{\sqrt{a(\varepsilon)\prt{1-a(\varepsilon)}}-\varepsilon}^2}$ goes
to 0 monotonically when $\varepsilon\to0$. In particular, for
$\varepsilon \approx 0$ we have $a(\varepsilon) \approx
\varepsilon^{2}$ and $\nu(\varepsilon) \approx 4 \varepsilon^{2}$.
This model can thus be used
to study the influence of the gap on the convergence of the two
algorithms.

\subsubsection*{Influence of $\varepsilon$ on the convergence}
We run Algorithm~\ref{algo:grad} and Algorithm~\ref{algo:scf} with fixed $\beta$
on this system. We start from a random point on the manifold $\Mc$.
We take as convergence
criterion $\normF{P^{k}-P(a(\varepsilon))}~\lq~10^{-12}$, and consider the algorithm has failed if convergence was not achieved after $50,000$ iterations.

On Figure~\ref{fig:cvg}, we plotted the number of iterations to achieve
convergence for each algorithm as a function of $\varepsilon$ (without changing the starting point), for two
different values of $\beta$: $10^{-1}$ and $10^{-3}$. The results
confirm the theory we developed in Section~\ref{sec:com}: the gap has a
strong influence on the convergence behavior of the SCF algorithm.   Indeed, as
the gap decreases, smaller and smaller damping parameters must be
used, and the number of iterations increases. In fact for this system,
$1 + \Op^{-1} K_{*}$ has a single eigenvalue equal to
$1 + \frac{2}{\nu(\varepsilon)} \approx 1 + \frac{1}{2\varepsilon^{2}}$
for $\varepsilon$ small. Thus we expect convergence for
$\beta < 4 \varepsilon^{2}$, and therefore a critical
$\varepsilon_{\rm c}$ of $\approx 0.158$ for $\beta = 10^{-1}$ and
$0.0158$ for $\beta = 10^{-3}$, with a number of iterations
proportional to $\frac 1 {\varepsilon - \varepsilon_{\rm c}}$ when
$\varepsilon > \varepsilon_{\rm c}$. The numerical results are in perfect
agreement with this prediction. By contrast, the gradient algorithm is
much less affected by the smallness of the gap, and converges in an essentially constant number of iterations: our
prediction for the convergence rate of that method is
$r = 1 - \beta(\nu(\varepsilon) + 2) \approx 1 - 2\beta$ for
$\varepsilon$ small, and therefore a number of iterations for
convergence to $10^{-12}$ of $124$ for $\beta = 10^{-1}$ and
$1.3 \times 10^{4}$ for $\beta = 10^{-3}$, again in perfect agreement
with the numerical results.

\begin{figure}[h!]
  \subfigure[Number of iterations to reach convergence for both algorithms
  as a function of $\varepsilon$ for $\beta=10^{-1}$.]{
    \includegraphics[scale = 0.8]{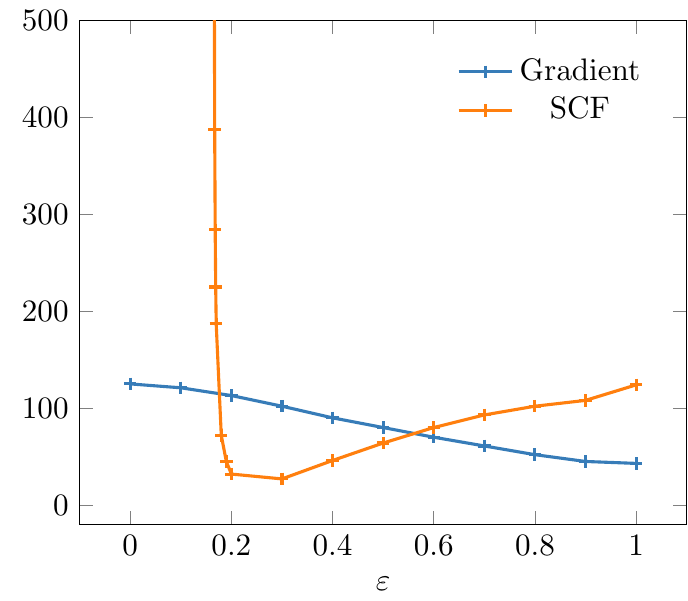}
    \label{fig:beta_10-1}}
  \subfigure[Number of iterations to reach convergence for both algorithms
   as a function of $\varepsilon$ for $\beta=10^{-3}$. On the left is
    a global view of the convergence for {$\varepsilon\in[0,1]$} and on the
    right, we zoom in the neighbourhood of $\varepsilon_{\rm c}$.  ]{
    \includegraphics[scale = 0.8]{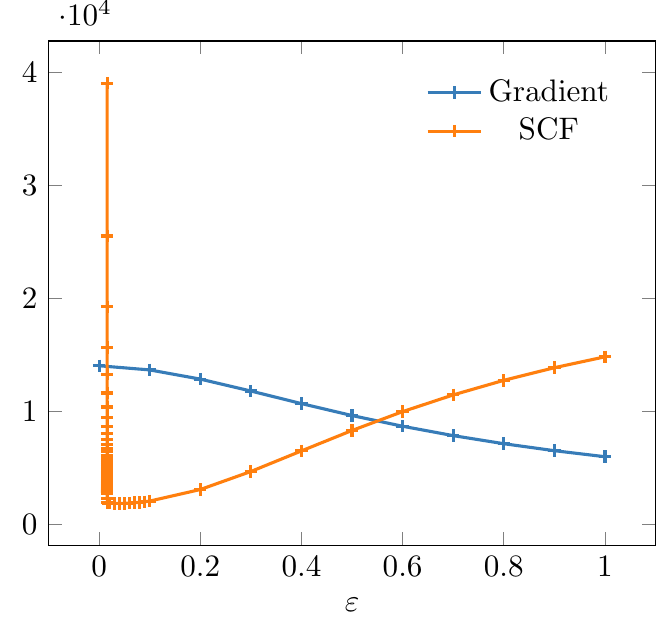}
    \hfill
    \includegraphics[scale = 0.8]{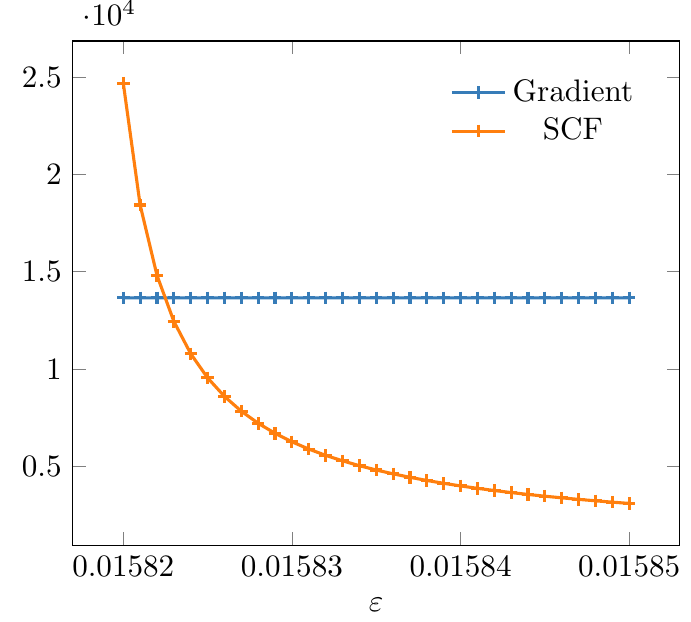}
    \label{fig:beta_10-3}}
  \caption{Comparison of the convergence of both algorithms depending on
    $\varepsilon$ for two different values of the step $\beta$.}
  \label{fig:cvg}
\end{figure}

\subsection{Chaos in SCF iterations} \label{sec:chaos}
We consider in this section another toy model a
model inspired from the one proposed in
\cite[Section 2.1]{liuAnalysisDiscretizedKohn2015}. Let $h \in \R^{3 \times 3}_\text{sym}$ and $J\in \R^{3 \times 3}$ be the matrices defined by
\[
  h \coloneqq \begin{bmatrix}
    1.4299 & -0.2839 & -0.4056 \\
    -0.2839 & 1.1874 & 0.2678 \\
    -0.4056 & 0.2678 & 2.3826 \\
  \end{bmatrix} \quad \mbox{and} \quad  J \coloneqq h^{-1}
\]
and the energy functional $E_{c_1,c_2} : \R^{N_b \times N_b}\to \R$ defined by
\begin{equation}\label{eq:Echaos}
E_{c_1,c_2}(P) = \Tr(hP) +  \frac {c_1}2 \rho_P^* J \rho_P- c_2 \sum_{j=1}^3 \rho_{P,j}^{4/3},
\end{equation}
where $c_1,c_2 \in \R_+$ are nonnegative real parameter and where the density
$\rho_P \in \R^3$ associated with the density matrix $P$ is given by
$\rho_{P,j}=P_{jj}$ for $j=1,2,3$. This model is reminiscent of a
Kohn-Sham LDA model, with $h$ playing the role of the core
Hamiltonian, $J$ of the Hartree operator and the third term in $E_{c_1,c_2}$
of the exchange-correlation energy. We seek the minimizers of $E_{c_1,c_2}$ on
$\Mc_1$.

We study the behavior of the simple SCF (Roothaan) iterations $P^k=\auf(P^k)$ with initial guess
$P^0=\phi_0\phi_0^*$ with $\phi_0$ a random vector of norm 1.

\subsubsection*{The case $c_2=0$}
Here, the energy functional is the sum of a linear and a quadratic
term. In this case,
either $(P^k)_{k \in \N}$ converges to a critical point of the problem (in
practice a local minimizer), or it has two different accumulation points
$P_{\rm odd}$ and $P_{\rm even}$, none of them being a critical point, and the
iterates oscillates between the two, in the sense that
$P^{2k+1} \to P_{\rm odd}$ and $P^{2k} \to P_{\rm even}$ when
$k\to\infty$ \cite{cancesConvergenceSCFAlgorithms2000c,levittConvergenceGradientbasedAlgorithms2012b}.
This is due to the fact that we have
\begin{align*}
P^{2k+1}&=\mbox{argmin} \left\{ \widetilde E_{c_1,0}(P^{2k},P), \; P \in \Mc_1 \right\}, \\
P^{2k+2}&=\mbox{argmin} \left\{ \widetilde E_{c_1,0}(P,P^{2k+1}), \; P \in \Mc_1 \right\},
\end{align*}
with
$$
\widetilde E_{c_1,0}(P,P')\coloneqq \frac 12 \Tr(hP) + \frac 12 \Tr(hP') +  \frac {c_1}2 \rho_P^* J \rho_P',
$$
so that $(P^{2k},P^{2k+1})$ converges to a minimizer of
$\widetilde E_{c_1,0}$ on $\Mc_1 \times \Mc_1$. When $c_1$ is small,
the simple SCF algorithm converges: for $c_{1} = 0$, the matrix $h$
has a nondegenerate lowest eigenvalue and the algorithm converges in
one iteration. When the value of $c_1$ increases, we observe
numerically a bifurcation at a critical value $c_{1,*} \simeq 0.28$
after which the algorithm oscillates between two states.

\subsubsection*{The case $c_2=1$}
Here the energy is not quadratic, and the previous theory does not
apply.
In Figure~\ref{fig:chaos}, we vary $c_1$ and plot the value of $\rho_1$ for the
last 40 out of 1,500 SCF iterations.
For this case, we still observe that the algorithm converges for $c_1$ small
enough ($0 \le c_1 < c_{1,*} \simeq 1.38$), and oscillates between two states for $c_1$
slightly larger than $c_{1,*}$. However, in contrast with the $c_2=0$
case, this is followed by a cascade of cycles of increasing periods,
transitioning to a chaotic region, following the ``period-doubling
route to chaos'' observed for other types of chaotic systems such as
the logistic map \cite{strogatz2001nonlinear}.

\begin{figure}[h!]
  \hfill
  \includegraphics[scale = 0.8]{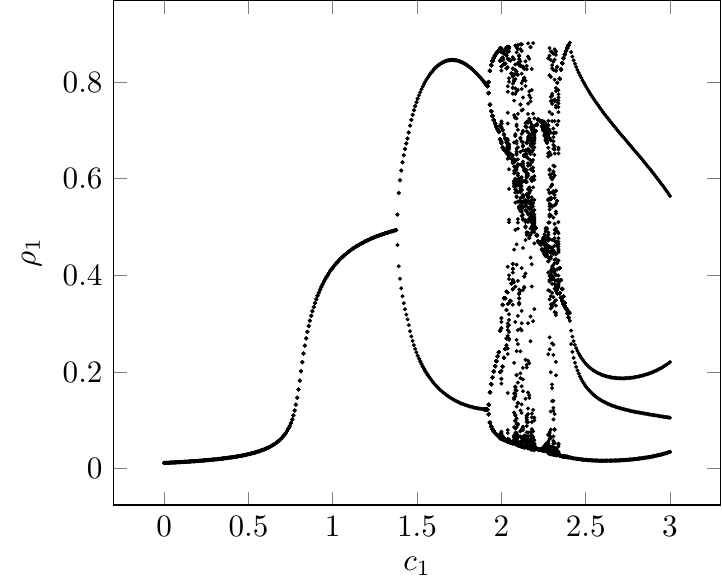}
  \hfill
  \includegraphics[scale = 0.8]{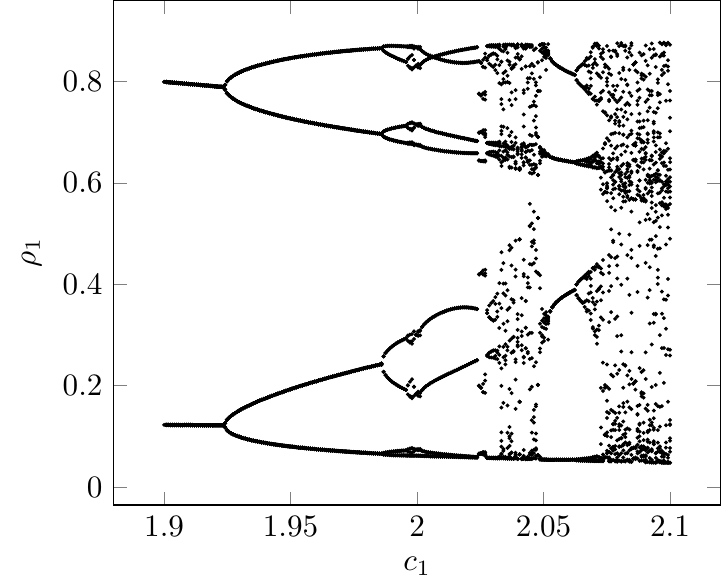}
  \hfill{~}
\caption{Chaotic behavior of the simple SCF map for the energy functional
  $E_{c_1,1}$ defined by \eqref{eq:Echaos} and $N=1$ as a function of $c_1$.
  On the left is a global view of the bifurcation and the right is a zoom on
  part of the interval on which we observe a chaotic behavior.}
\label{fig:chaos}
\end{figure}

\subsection{Local convergence for a 1D nonlinear Schr\"odinger equation}
\label{sec:validate}
In this section, we present a simple 1D numerical experiment to validate on a
more physically relevant system the
sharpness of the convergence rates we derived in the previous section. As the
goal here is to analyze the behavior of the simplest representative of each
class of algorithms when physical parameters (such as the gap) vary,
we chose unpreconditioned algorithms.
We consider a discretized 1D Gross-Pitaevskii model
$(N=1)$ on the torus, and its (non-physical) fermionic counterpart for $N=2$. At the continuous level, the minimization set is
$$
\left\{ \gamma \in {\mathcal L}(L^2_{\rm per}), \; \gamma^2=\gamma=\gamma^*, \; \Tr(\gamma)=N \right\},
$$
where ${\mathcal L}(L^2_{\rm per})$ is the space of bounded operators on $L^2_{\rm per}\coloneqq\{u \in L^2_{\rm loc}(\R) \; | \; u(\cdot-1)=u(\cdot) \}$, and the energy functional is defined as
$$
{\mathcal E}_\alpha(\gamma) = {\Tr}_{L^2_{\rm per}} \left( - \frac 12 \Delta \gamma \right) + \int_0^1 V \rho_\gamma + \frac{\alpha}2 \int_0^1 \rho_\gamma^2,
$$
where $\rho_\gamma$ is the density of the density matrix $\gamma$,
$\alpha \in \R_+$ and $V$ is an asymmetric double-well external potential chosen equal to
\begin{equation}
  \label{eq:potential}
V(x) \coloneqq -C\left(\exp\left(-c\cos\left(\pi(x-0.20)\right)^2\right) +
2\exp\left(-c\cos\left(\pi(x+0.25)\right)^2\right)\right),
\end{equation}
with $c = 30$ and $C = 20$ (Figure~\ref{fig:V}).
\begin{figure}[h!]
  \includegraphics[scale = 0.8]{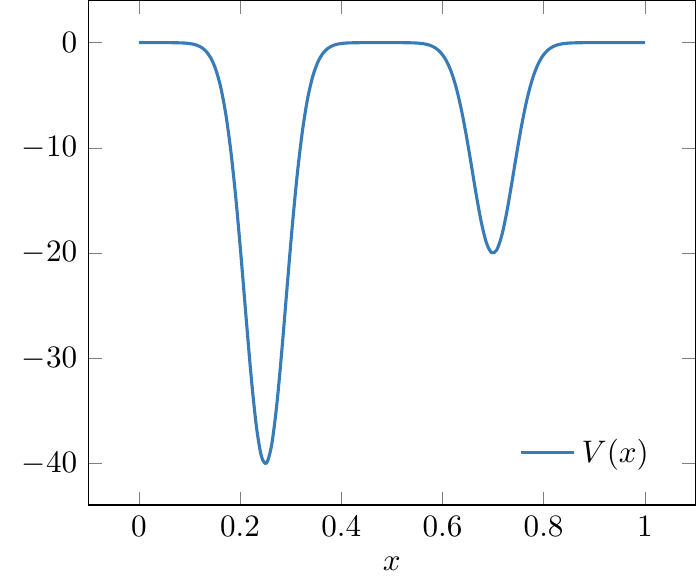}
  \caption{$V$ for $c=30$ and $C=20$.}
  \label{fig:V}
\end{figure}

The Euler-Lagrange equations of this minimization problem are
\begin{equation}
  \label{eq:GP}
 \gamma_*=\sum_{i=1}^N (\phi_i,\cdot)_{L^2_{\rm per}} \phi_i, \quad \rho_* = \sum_{i=1}^N|\phi_i|^2, \quad  -\frac{1}{2}\Delta\phi_i + V\phi_i + \alpha\rho_* \phi_i = \varepsilon_i\phi_i, \quad \int_{0}^1 {\phi_i}\phi_j = \delta_{ij}.
\end{equation}
We discretize this model using the finite difference method with a uniform
grid of step size $\delta=1/N_b$, which leads to the
finite-dimensional model
\begin{equation}\label{eq:FDNLS}
\inf\left\{ E_\alpha(P), \; P \in \Mc_N \right\},
\end{equation}
where
\begin{equation}\label{eq:discr_nrj}
  \Forall P\in\Hc,\quad E_\alpha(P)\coloneqq
  \Tr(hP) + \frac{\alpha}{2} \delta\sum_{i=1}^{N_b}\prt{\frac{P_{ii}}{\delta}}^2,
\end{equation}
the nonzero entries of the matrix $h \in \R^{N_b \times N_b}$ being given by
$$
\Forall 1\le i \le N_{b}, \quad h_{ii}=\frac{1}{\delta^2}+V(i\delta), \quad h_{i,i+1}=h_{i,i-1}=-\frac{1}{2\delta^2}.
$$
where we identify the sites $0$ and $N_{b}$ on the one hand and
$N_{b}+1$ and $1$ on the other. With this discretization, the discrete
density is then given by
$\rho(i\delta) \approx \rho_i \coloneqq P_{ii}/\delta$. We compare
the fixed-step gradient descent and damped SCF algorithms
(Algorithm~\ref{algo:grad} and Algorithm~\ref{algo:scf}) on this problem for
various values of $\alpha$, using as starting point the ground state
for $\alpha=0$. The functional $E$ is smooth. To check Assumption~\ref{asmp2}
we notice that $E$ is a convex functional of $P$, so that
$\nabla^{2} E(P) \geqslant 0$. Therefore, at any local minimizer
satisfying the strong \emph{Aufbau} principle,
$\Op + K_{*} \geqslant \Op \ge \eta > 0$ and therefore Assumption~\ref{asmp2}
is satisfied. In the case where the \emph{Aufbau} principle is not
satisfied, Assumption~\ref{asmp2} is not {\it a priori} always satisfied, so
we check it \emph{a posteriori} by computing the lowest eigenvalue of
$\Omega_{*} + K_{*}$.

We prove in the Appendix the following lemma, which collects some of the
mathematical properties of the discretized model under consideration.
The proof of this lemma, given in the appendix, strongly relies on the
properties of our particular model (one-dimensional difference
equation with periodic boundary conditions and a specific nonlinearity).

\begin{lemma}[Mathematical properties of  \eqref{eq:FDNLS}] \label{lem:NLS} Let $\alpha \in \R_+$.
\begin{enumerate}
  \item For $N=1$, the optimization problem \eqref{eq:FDNLS} has a unique
    minimizer $P_*$. In addition, $P_*$ can be written as $P_*=\phi_*\phi_*^*$,
    with $\phi_* \in \R^{N_b}$, $\phi_*^*\phi_*=1$, and $\phi_*$ positive
    componentwise, and $P_*$ satisfies the strong {\em Aufbau} principle.
  \item For $2 \le N \le N_b$, with $N_b \ne 2N$ if $N_b \in 4\N^\ast$, the relaxed constrained optimization problem
\begin{equation} \label{eq:relaxed}
\inf\left\{ E_\alpha(P), \; P \in {\rm CH}(\Mc_N) \right\},
\end{equation}
where ${\rm CH}(\Mc_N)=\{P \in \Hc, \; P=P^*, \; 0 \le P \le 1, \, \Tr(P)=N \}$
is the convex hull of $\Mc_N$, has a unique minimizer $P_*$. Either
$P_* \in \Mc_N$, in which case $P_*$ is the unique minimizer of
\eqref{eq:FDNLS} and satisfies the {\em Aufbau} principle, or
$P_* \notin \Mc_N$, in which case the eigenvalues
$\varepsilon_1 \le \cdots \le \varepsilon_{N_b}$ of the mean field Hamiltonian
matrix $H_*=\nabla E_\alpha(P_*)$ satisfy
$\varepsilon_{N-1} < \varepsilon_{N} = \varepsilon_{N+1} < \varepsilon_{N+2}$
and none of the local minimizers to \eqref{eq:FDNLS} satisfies the {\em
  Aufbau} principle.
\end{enumerate}
\end{lemma}
Note that the unique minimizer $P_*$ to the relaxed constraint problem \eqref{eq:relaxed} can be computed using
the optimal damping algorithm (ODA) introduced
in~\cite{cancesCanWeOutperform2000b}. As shown in the proof of
Lemma~\ref{lem:NLS}, the only case when the minimizer $P_*$ is not
unique is the very particular case when $N_b \in 4\N^\ast$, $N_b=2N$, and all
the entries $[V_{\rm eff}]_i\coloneqq V(i\delta)+\alpha \delta^{-1}[P_*]_{ii}$
of the effective potential are equal. In the rest of this section, we consider
the cases $N=1$ and $N=2$.

\subsubsection*{The case $N=1$} It follows from
Lemma~\ref{lem:NLS}, Theorem~\ref{thm:grad} and Theorem~\ref{thm:scf} that
Algorithm~\ref{algo:grad} and Algorithm~\ref{algo:scf} locally converge to the
unique minimizer $P_*$ as long as $\beta$ is chosen small enough.
The resulting densities, effective potentials and convergence behavior of both algorithms are plotted in
Figure~\ref{fig:num_res} for $N_b = 100$.
The SCF algorithm converges faster in terms of number of iterations, as a smaller $\beta$, hence
more steps, are required for the gradient to converge. This is expected
from the large spectral radius of the matrix $h$ in the absence of preconditioning.

For the gradient algorithm, the convergence rate is consistent with
the spectral radius of the Jacobian matrix $1 - \beta J_\text{grad}$.
For the damped SCF algorithm with the ground state of the core
Hamiltonian as starting point, surprisingly, we observe an asymptotic
convergence rate slightly faster than that expected from the spectral
radius of the Jacobian matrix $1 - \beta J_\text{SCF}$. Using a random
perturbation of the ground state of the core Hamiltonian as starting
point again gives a convergence rate consistent with the spectral
radius.

The explanation of this effect is to be found in the repartition of
the error among the eigenvectors of $J_{\text{SCF}}$. Since both $\Op$
and $K_*$ are positive semidefinite operators, the eigenvalues of
$J_\text{SCF} = 1 + \Op^{-1}K_*$ are greater than 1, and the
convergence for $\beta$ small is limited by the modes associated with
eigenvalues of $J_{\text{SCF}}$ close to 1. These eigenvalues
correspond to high eigenvalues of $\Op$, and therefore to highly
oscillatory modes. When the initial guess is chosen as the ground
state of the core Hamiltonian, these modes are only weakly excited and
do not contribute to the observed convergence rate before convergence
is achieved. When the initial guess is randomly perturbed, this effect
is not present and the convergence rate is consistent with the
spectral radius. For the gradient algorithm, the rate-limiting modes
are associated with small eigenvalues of $\Op + K_{*}$, which are not
oscillatory, and this effect is not present either.

\vfill
\begin{figure}[h!]
  \hfill
  \subfigure[Density $\rho$ at convergence.]{
      \includegraphics[scale = 0.8]{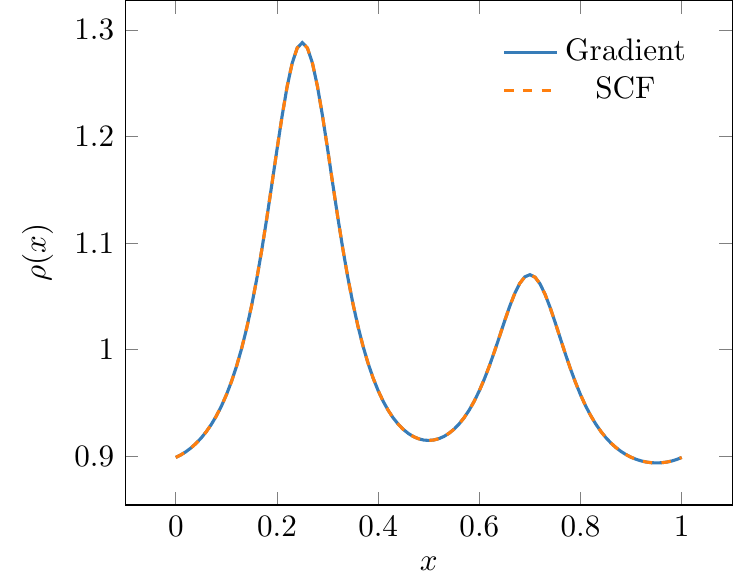}
  }
  \hfill
  \subfigure[Effective potential $V_\text{eff} = V + \alpha\rho$ at convergence.]{
      \includegraphics[scale = 0.8]{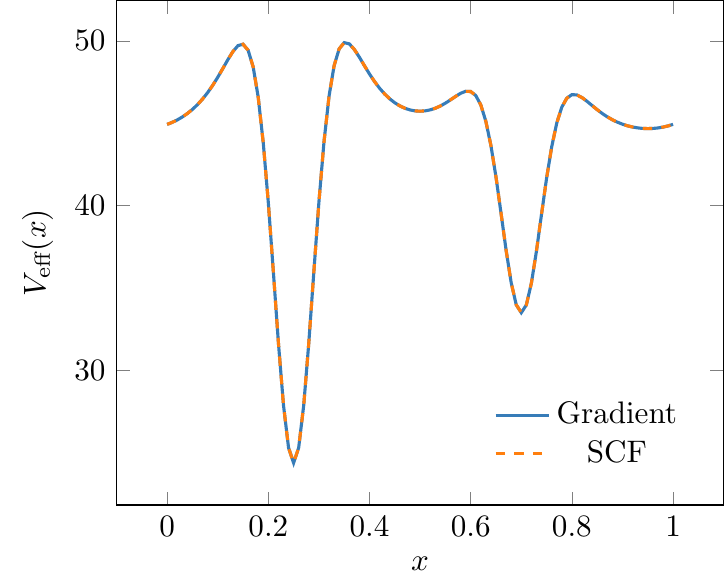}
  }
  \hfill{~}

  \subfigure[Error decay for SCF and gradient descent algorithms
  (every few mark only is plotted for the sake of visibility). Marked lines are
  the evolution of the error $\normF{P^{k+1} - P^k}$ and dashed lines represents
  the slope computed with the spectral radius of the Jacobian matrix, computed
  by finite differences.]{
    \includegraphics[scale = 0.8]{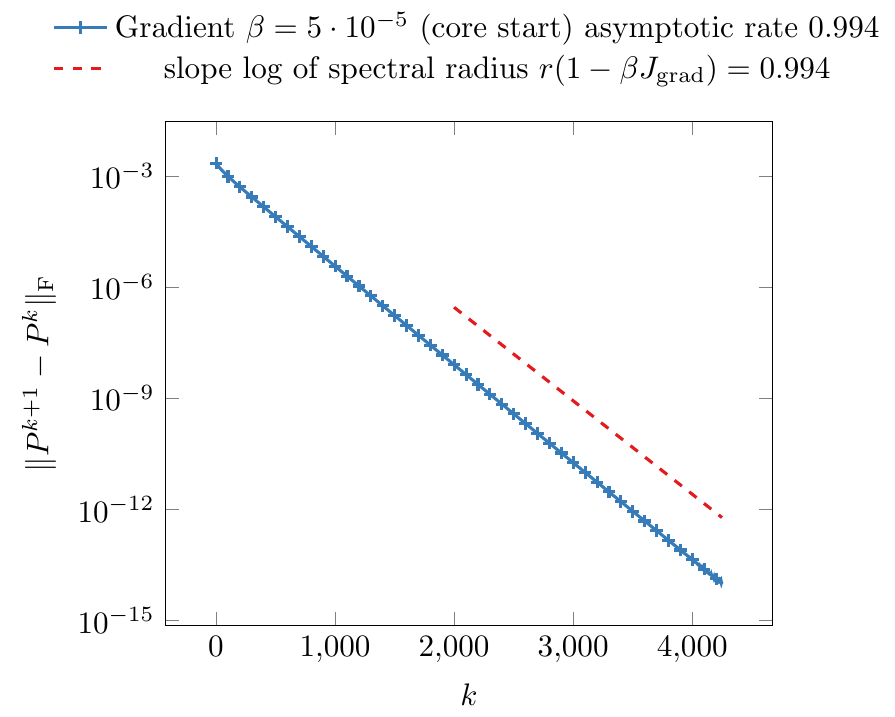}
    \hfill
    \includegraphics[scale = 0.8]{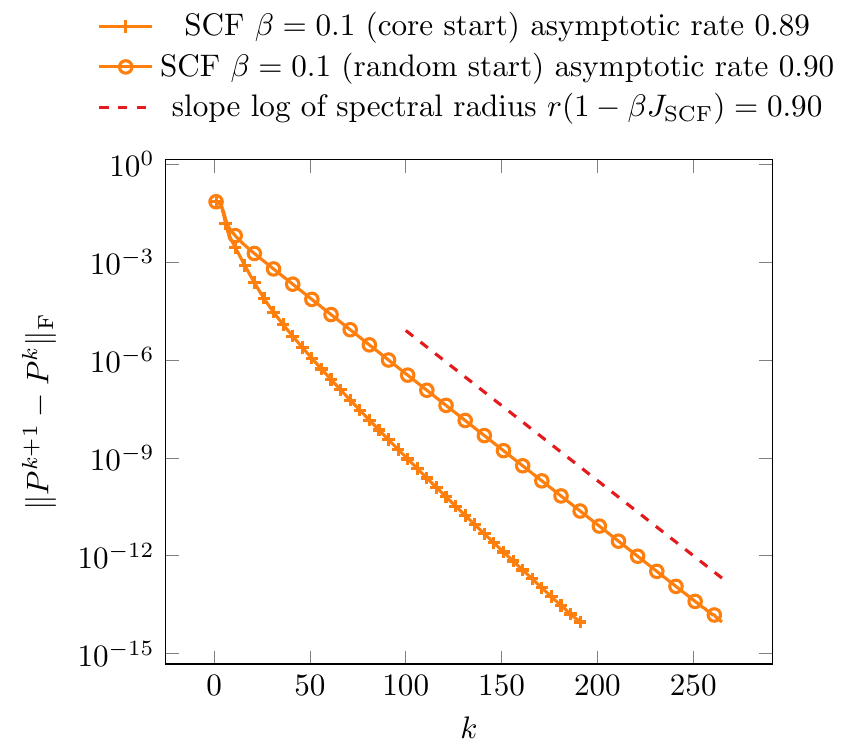}
  }
  \caption{Convergence of Algorithm~\ref{algo:grad} and Algorithm~\ref{algo:scf}
    for $N=1$, $\alpha = 50$ and $N_b = 100$.}
  \label{fig:num_res}
\end{figure}
\vfill

\clearpage
\subsubsection*{The case $N=2$}
Since the second smallest eigenvalue of the matrix $h$ is
strictly lower than the third one, for $\alpha$ small enough, the
unique minimizer $P_*$ of \eqref{eq:relaxed} is on $\Mc_2$ and
satisfies the strong {\em Aufbau} principle, and both the gradient
descent and SCF algorithm locally converge to $P_*$. For larger values of
$\alpha$, the two alternatives of Lemma~\ref{lem:NLS} appear. We plot the energy, the
density at an arbitrarily chosen point and the eigenvalues of the
solutions $P^{\rm grad}$ and $P^{\rm ODA}$
obtained by the gradient and the ODA algorithm as a function of
$\alpha$ in Figure~\ref{fig:bifur}, evidencing a bifurcation for a
critical value of $\alpha_{\rm c}\simeq 10$, after which these two
solutions are different.

\begin{figure}[h!]
  \hfill
  \subfigure[Bifurcation on the energy.\label{fig:bif_nrj}]{
      \includegraphics[scale = 0.8]{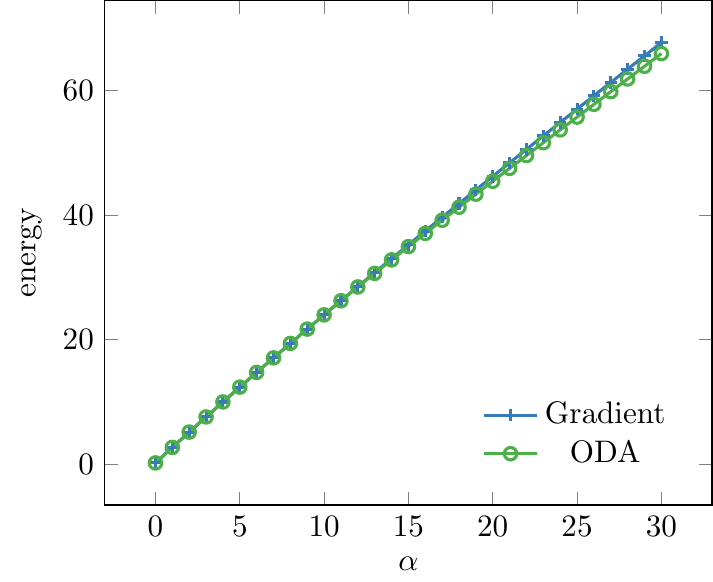}
  }
  \hfill
  \subfigure[Bifurcation on the value of $\rho_7$ ($7^{\rm th}$ element of the
  discrete density for $N_b = 40$).\label{fig:bif_rho}]{
      \includegraphics[scale = 0.8]{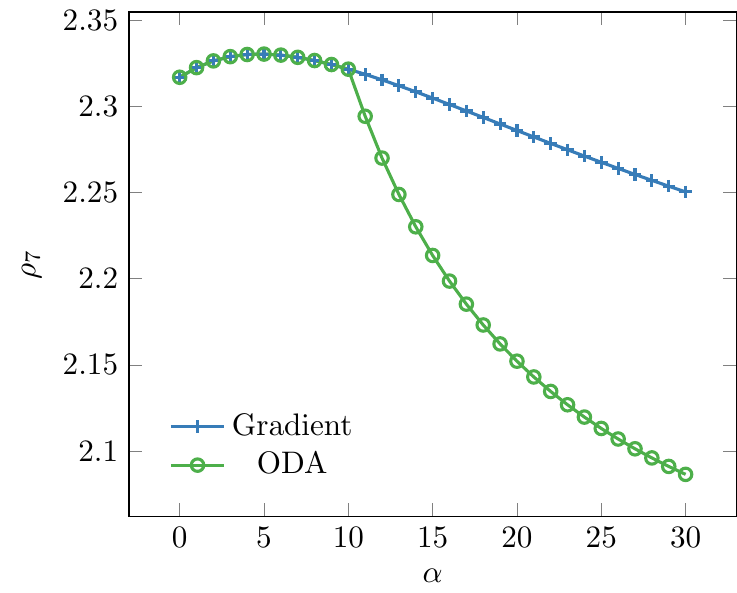}
  }
  \hfill{~}

  \hfill
  \subfigure[Bifurcation on the first three eigenvalues of the Hamiltonians
  $H(P_*^\text{ODA})$ and $H(P_*^\text{grad})$. For
  ODA, after the collision, the second and third eigenvalues are
  degenerate.
  \label{fig:bif_eps}]{
      \includegraphics[scale = 0.8]{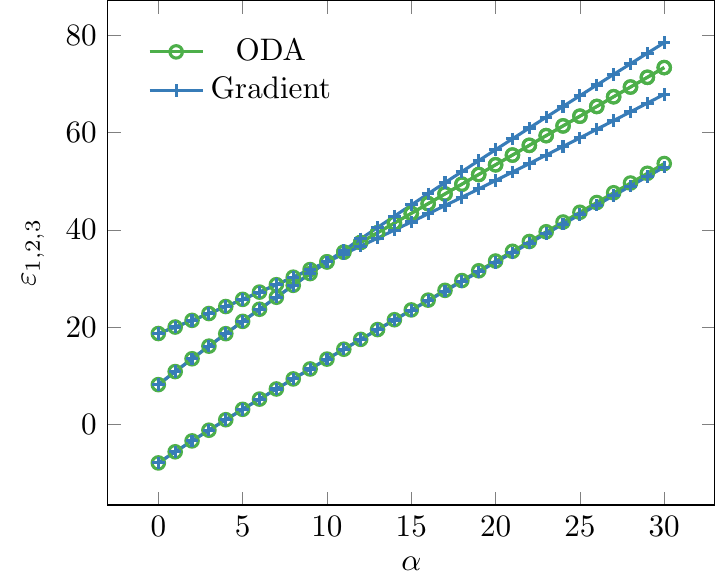}
  }
  \hfill
  \subfigure[ODA minimum before (left) and after (right) the bifurcation.
  Dashed lines are the energy level sets around the minimum.
  \label{fig:bifur_oda}]{\includegraphics[scale = 0.85]{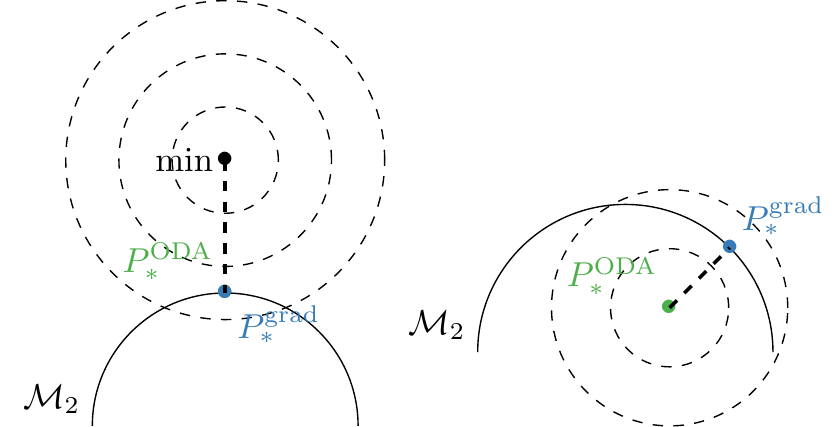}}
  \hfill{~}
  \caption{Bifurcation on the energy, the density and the eigenvalues
    as a function of $\alpha$
    for $N=2$, $N_{b}=40$.}
  \label{fig:bifur}
\end{figure}

For $\alpha$ lower than the critical value $\alpha_{\rm c}\simeq 10$,
$P_*^{\rm ODA}$ is on the manifold $\Mc_2$ and satisfies the strong
{\em Aufbau} principle. The algorithms all converge to this solution:
$P_{*}^{\rm grad} = P_{*}^{\rm SCF} = P_{*}^{\rm ODA} = P_{*}$.
However for $\alpha > \alpha_{\rm c}$ in the range tested,
$P_*^{\rm ODA} \notin \Mc_2$. A geometrical interpretation of the
bifurcation is sketched on Figure~\ref{fig:bifur_oda}: the level sets of
the function $E_{\alpha}$ are degenerate ellipsoids. Below
$\alpha_{\rm c}$, the intersection of the nonempty closed convex set
${\rm CH}(\Mc_2)$ with the level set of $E_{\alpha}$ of lowest energy
belongs to $\Mc_2$, while this is no longer the case beyond
$\alpha_{\rm c}$.

For $\alpha > \alpha_{\rm c}$, the solutions obtained by the ODA,
gradient and SCF algorithm differ, as shown in Figure~\ref{fig:diff_sol}:
\begin{itemize}
  \item the lowest second and third eigenvalues of $H(P_*^{\rm ODA})$ are
    degenerate ($\varepsilon_2=\varepsilon_3$) and $P_*^{\rm ODA} \notin \Mc_2$.
    More precisely, we have
$$
P_*^{\rm ODA}=\phi_1\phi_1^* + (1-f) \phi_2\phi_2^* + f \phi_3\phi_3^* \quad
\mbox{with} \quad H(P_*^{\rm ODA})\phi_i=\varepsilon_i\phi_i, \quad
\phi_i^*\phi_j=\delta_{ij}
$$
with a fractional occupation $0 < f < 1$;
\item Algorithm~\ref{algo:grad} and Algorithm~\ref{algo:scf} converge to two different limits $P^{\text{grad}}_*$ and
$P^{\text{SCF}}_*$, none of them satisfying the {\em Aufbau} principle:
\begin{itemize}
  \item  $P^{\text{grad}}_*$ is a local minimizer of $E_\alpha$ on $\Mc_2$,
    which does not satisfies the {\em Aufbau} principle. More precisely,
    $P^{\text{grad}}_*$ is the orthogonal projector on the space generated by
    the eigenvectors associated with the lowest first and third eigenvalues of
    $H(P^{\text{grad}}_*)$;
  \item $P_*^\text{SCF}$ satisfies
    $\Pi_{P_*^\text{SCF}}\prt{\auf(P_*^\text{SCF})-P_*^\text{SCF}} =
    0$, but $\auf(P_*^\text{SCF}) - P_*^\text{SCF} \neq 0$ and
    $[H(P_*^\text{SCF}),P_*^\text{SCF}]~\neq ~0$. The iterates are
    trapped as the search direction is orthogonal to the tangent space
    (Figure~\ref{fig:diff_sol_draw}). The limit point $P_*^\text{SCF}$
    is a spurious stationary state of the SCF iteration which is not
    physically relevant, not being a critical point of $E_\alpha$.
\end{itemize}
\end{itemize}

\begin{figure}[h!]
  \hfill
  \subfigure[Densities of $P^{\text{grad}}_*$, $P^{\text{SCF}}_*$ and $P^{\text{ODA}}_*$.\label{fig:diff_sol_plot}]{
    \includegraphics[scale = 0.8]{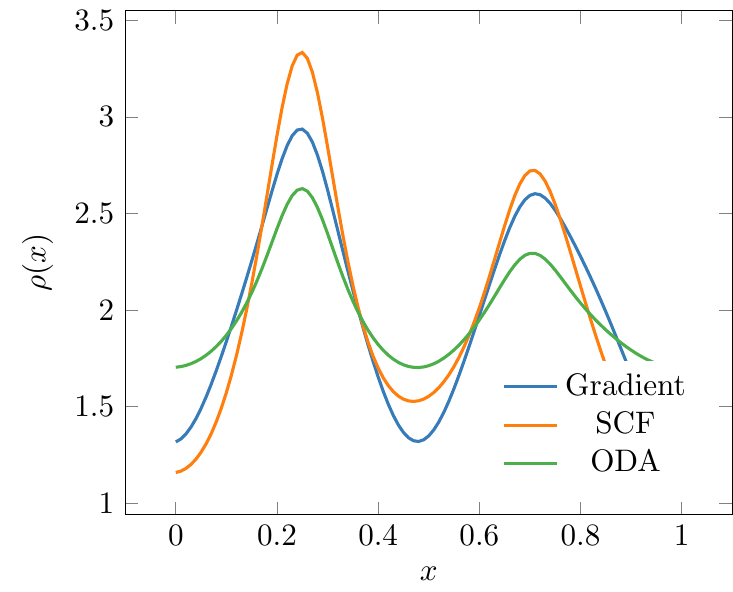}
  }
  \hfill
  \subfigure[Effective potential $V_\text{eff} = V + \alpha\rho_{P}$ for $P=P^{\text{grad}}_*$, $P^{\text{SCF}}_*$ and $P^{\text{ODA}}_*$.
  \label{fig:diff_Veff_plot}]{
    \includegraphics[scale = 0.8]{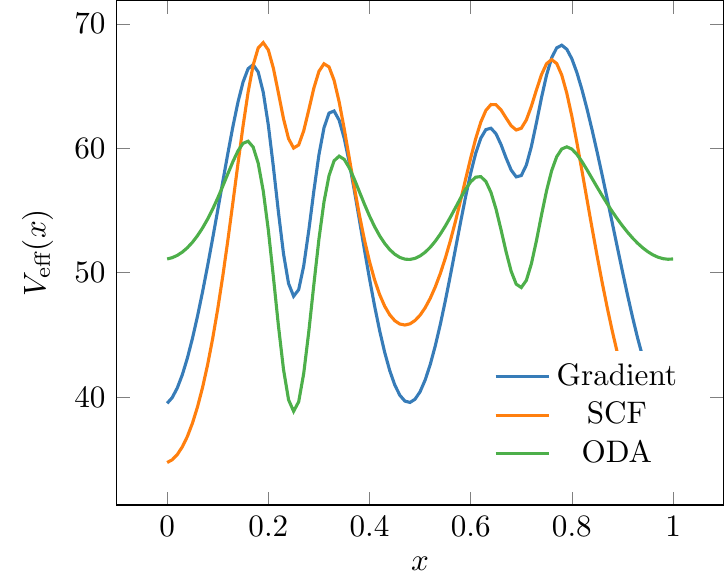}
  }
  \hfill{~}

  \hfill
  \subfigure[Geometrical interpretation of the limiting points of the gradient descent, SCF et ODA algorithms.
  \label{fig:diff_sol_draw}]{
    \includegraphics[scale = 0.8]{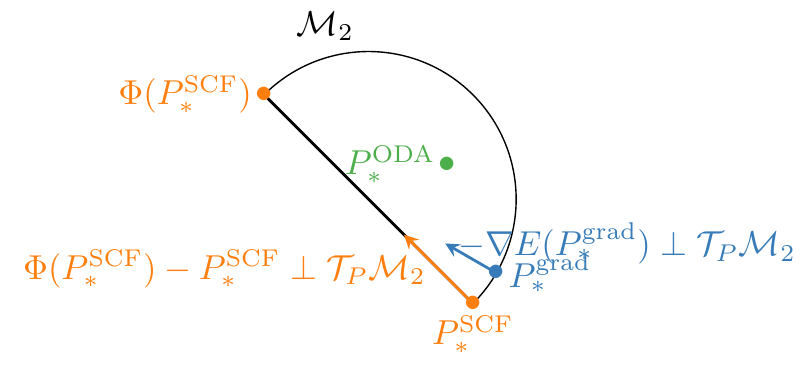}
  }
  \hfill
  \subfigure[Lowest three eigenvalues of $H(P)$ for $P=P_*^{\rm
    grad}$, $P=P_*^{\rm SCF}$ and $P=P_*^{\rm ODA}$.
  \label{fig:diff_sol_eig}]{
    \includegraphics[scale = 0.8]{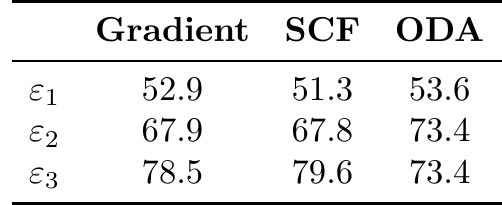}}
  \hfill{~}

  \caption{Results obtained with the gradient descent, damped SCF and
    ODA algorithm for $N=2$, $\alpha=30$ and $N_b = 100$: the limiting
    points $P_*^{\rm grad}$ and $P_*^{\rm SCF}$ lay by construction on the
    manifold $\Mc_2$, while $P_*^{\rm ODA}$ does not (it only belongs to its
    convex hull ${\rm CH}(\Mc_2)$).}
  \label{fig:diff_sol}
\end{figure}

\subsection{Kohn-Sham density functional theory}
\label{sec:DFTK}
We now investigate a more realistic computation: the electronic structure of a Silicon crystal
using Kohn-Sham density functional theory (KS-DFT). We used the \texttt{DFTK.jl}
code~\cite{michael_f_herbst_2020_3703064}, which solves the equations
of KS-DFT in a plane-wave basis under a pseudopotential approximation.
All computations below use the local density approximation (LDA) of the
exchange-correlation energy \cite{Kohn1965,martin2004}, Goedecker-Teter-Hutter (GTH)
pseudopotentials \cite{goedecker1996separable}, a cutoff energy of 30 Hartrees, and
$\Gamma$-only Brillouin zone sampling for simplicity, although the
same behavior was observed with different exchange-correlation
functionals and fine Brillouin zone discretizations. In all cases, the
initial guess for the algorithms is a superposition of atom-centered
densities. The DFTK code as well as the script used to produce these
results are available at \url{https://dftk.org/}.

We consider the case of Silicon in its standard face-centered cubic
phase. With the chosen pseudopotentials and without spin polarization,
Silicon has four occupied orbitals: $N=4$. We examine the convergence
of algorithms as a function of the lattice constant $a$ (the size of
the computational domain). In the equilibrium state of Silicon ($a \approx 10.26$
Bohrs), there is a gap of about $0.08$ Hartree between the occupied
and virtual states. As the lattice constant $a$ is increased, this gap
decreases, until it closes at $a \approx 11.408$ Bohrs. We examine the
convergence of self-consistent iterations, using both fixed-step
damped density mixing with four values of the mixing parameter
$\beta$ ($\beta=1$ -- no damping --, $\beta=0.5$, $\beta=0.2$,
$\beta=0.1$), as well as the self-consistent iteration accelerated
with Anderson acceleration (also known as Pulay's DIIS
method~\cite{chupin2020convergence,Pulay1980,Pulay1982}). We plot the
convergence of the density residual $\|\rho_{\auf(P_{n})} - \rho_{P_{n}}\|_{2}$ as a
function of the iterations for three values of $a$, with decreasing
gaps.

\begin{figure}[h!]
  \centering
  \hfill
  \includegraphics[scale = 0.8]{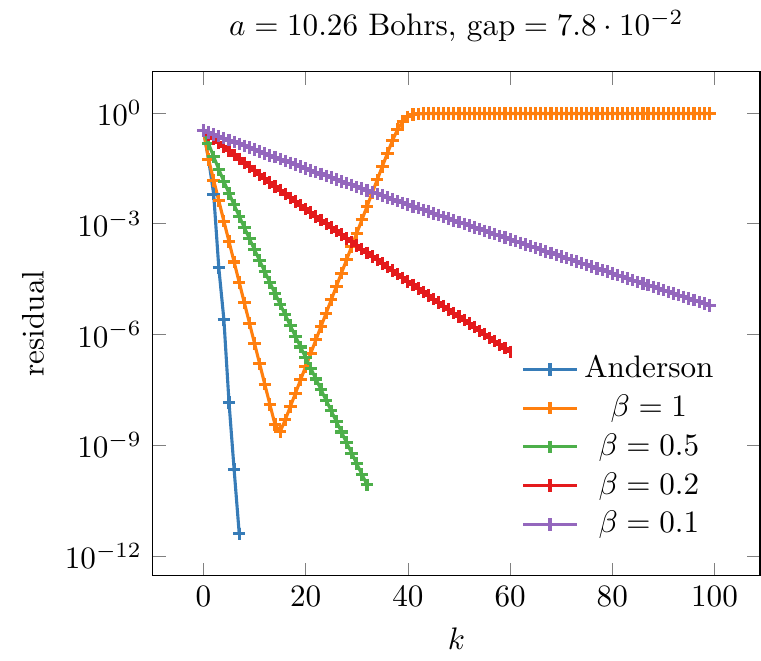}
  \hfill
  \includegraphics[scale = 0.8]{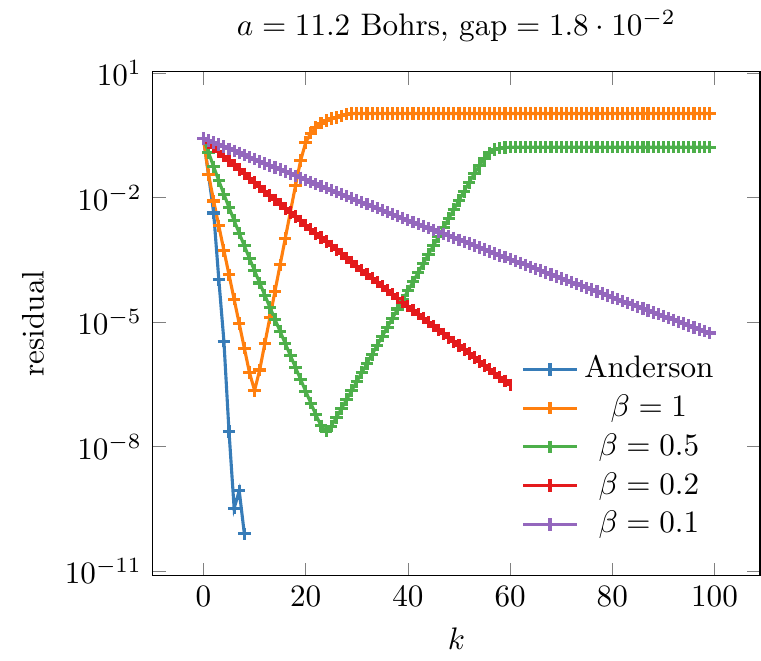}
  \hfill{~}

  \includegraphics[scale = 0.8]{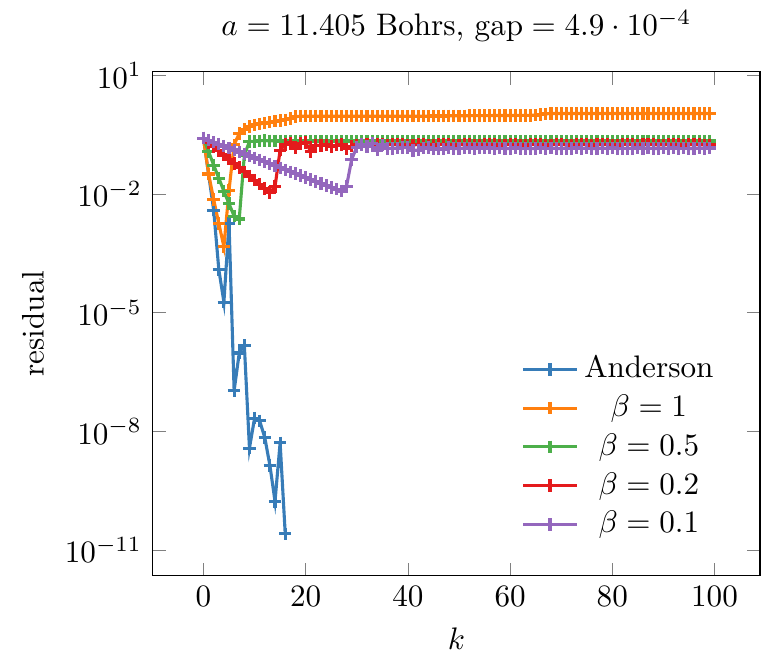}
  \caption{Convergence curves of the density residual as a function of the number of iterations $k$ for Silicon with different lattice
    constants $a$.}
  \label{fig:silicon}
\end{figure}

In the first case, with $a = 10.26$ Bohrs, the simple (undamped) SCF method appears
to be converging for almost 20 iterations, but then diverges, until
the density residual stabilizes at a positive value, as predicted
in~\cite{cancesConvergenceSCFAlgorithms2000c} for the Hartree-Fock
model. The damped methods appear to converge. For $a = 11.2$ Bohrs, the
damping method with $\beta = 0.5$ does not converge. When the lattice
constant is further increased to $11.405$ Bohrs, with a small gap of $4.9
\times 10^{-4}$ Hartree, the fixed-step damped SCF iterations do not converge
for the tested values of the damping parameter ($\beta=1, 0.5, 0.2, 0.1$).

When it occurs, the transient behavior of apparent-convergence before
eventual divergence is unusually long. For instance for $\beta=1$ at
$a=10.26$ Bohrs, the method appears to be converging for almost $20$
iterations, up to a reduction in residual of a factor $10^{-8}$. In
fact, it is consistent with an initial error of the order of machine
precision (about $10^{-16}$ here) being amplified at a constant rate.
The cause of this effect appears to be that the divergent modes break
the natural inversion symmetry of the crystal in this particular case:
we have checked that the divergence occurs much sooner if we break
this symmetry by perturbing the positions of the atoms around their
symmetric positions (at $9$ iterations by perturbing the position of
one atom by $10\%$).
In practice, in the symmetric case, one way to overcome this issue is to ensure
during the algorithm that, at each step, we have a symmetric solution.
Note that this phenomenon is reminiscent of that observed in Figure~\ref{fig:num_res}, where all the modes were not
fully excited, making the convergence faster than expected.

It is remarkable that the convergent methods (and even the divergent
ones before their divergence) appear to have the exact same slope as
with $a=10.26$ Bohrs. This is consistent with our result: assuming the main
effect of increasing the lattice constant is to decrease the gap,
while keeping the lowest eigenvalues of $\Omega_{*}^{-1} K_{*}$
constant, then for $\beta$ small enough the convergence is limited by
the lowest, not the highest, eigenvalues of this operator.

In all these cases, Anderson acceleration was able to converge to a
solution, even in presence of a very small gap, albeit in an irregular
fashion. We attribute this to the well-known fact that, in the linear
regime, Anderson acceleration is equivalent to the GMRES algorithm.
Since the GMRES algorithm is a Krylov method, it is robust to the
presence of a large eigenvalue, and achieves convergence even though
the underlying iteration is strongly divergent. This shows a limitation of our
theoretical convergence rates, which do not capture the reduced
sensitivity of accelerated methods to a small gap.

\section{Conclusion}
\label{sec:conclusion}
In this paper, we examined the convergence of two simple
representatives in the class of direct minimization and SCF
algorithms. We showed that both algorithms converge locally when the
damping parameter is chosen small enough. We derived their convergence
rates; we showed that the damped SCF algorithm is sensitive to the
gap, while the gradient method is sensitive to the spectral radius of
the Hamiltonian. We confirmed these results with numerical
experiments. The goal here was not to propose efficient algorithms, but to
analyze the behavior of the simplest representatives of each class.
However, accelerated algorithms are generally found to follow the trend suggested by our
theoretical results, although we showed that the Anderson-accelerated
SCF algorithm was able to converge quickly even in the presence of a
single very small gap.

In practice, should the SCF or direct minimization class of algorithms
be preferred? The answer depends not only on the convergence rate
studied in this paper, but also on the cost of each step, and the
robustness of the algorithm. We examine two prototypical situations.

In quantum chemistry using Gaussian basis sets to solve the
Hartree-Fock model or Kohn-Sham density functional theory using hybrid
functionals, the rate-limiting step is often the computation of the
Fock matrix $H(P)$. In this case, an iteration of a gradient descent
and a damped SCF algorithm are of roughly equal cost. In most cases,
solutions for isolated molecules satisfy the Aufbau principle, and the
damped SCF algorithm, suitably robustified (for instance using the ODA
algortithm) and accelerated (for instance with the DIIS algorithm),
converges reliably and efficiently towards a solution. Direct
minimization algorithms are then only useful in the cases where local
or semilocal functionals are used \cite{rudberg2012difficulties} and the Aufbau
principle is violated, or when SCF algorithms tend to converge to
saddle points (for instance for computations involving spin).

In condensed-matter physics using plane-wave basis sets to solve the
Kohn-Sham density functional theory with local or semilocal
functionals, the matrices $P$ and $H$ are not stored explicitly.
Solving the linear eigenproblem is then done using iterative block
eigensolvers, which can be understood as specialized direct
minimization algorithms in the case of a linear energy functional
$E(P) = \Tr(H_{0} P)$. In this case, direct minimization algorithms
effectively merge the two loops of the SCF and linear eigensolver, and
should therefore be more efficient. Another interest of direct
minimization algorithms is their robustness, as the choice of a
stepsize can be made in order to minimize the energy, unlike the
damped SCF algorithm where choosing an appropriate damping parameter
is often done empirically.

Despite this, direct minimization algorithms are rarely used in
condensed-matter physics. The main reason seems to be that challenging
problems are often metallic in character, and require the introduction
of a finite temperature. Direct minimization algorithms then need to
optimize over the occupations as well as the orbitals, a significantly
more complex task \cite{cances2001self,
  freysoldt2009direct,marzari1997ensemble}. A thorough comparison of the performance and
robustness of direct minimization and self-consistent approach for
these systems would be an interesting topic of inquiry. A number of
implementation ``tricks'' commonly used to accelerate the convergence
of iterative eigensolvers (for instance, using a block size larger
than the number of eigenvectors sought) might also play a large role
in performance comparison for the two classes of algorithms:
understanding how to generalize these to direct minimization would be
interesting.

We discussed in Remark~\ref{rmk:prec} preconditioning for both direct
minimization and SCF algorithms. The concept of preconditioning for
Riemannian optimization problems seems not to have been
explored much in the mathematical literature, except in some specific models
and preconditioners (see for instance
\cite{baoComputingGroundState2004,zhangExponentialConvergenceSobolev2019} for
the Gross-Pitaevskii model), and a deeper analysis of this would be interesting. In particular, this is
necessary to extend the convergence theory presented in this paper to
infinite-dimensional settings.

\section*{Appendix: proof of Lemma~\ref{lem:NLS}}
For any $\alpha \in \R_+$ the energy functional $E_\alpha$ is smooth
and convex on the nonempty compact convex set ${\rm CH}(\Mc_N)$. The set of
minimizers to \eqref{eq:relaxed} is therefore nonempty, compact and convex.
Let $P_*$ and $P_*'$ be two minimizers of $E_\alpha$ and let $\rho_*$ and
$\rho_*'$ be their densities: $\rho_{*,i} = \delta^{-1}
(P_{*})_{ii}$ and $\rho_{*,i}' = \delta^{-1} (P_{*}')_{ii}$. For all $\theta \in [0,1]$, we have
$$
I_\alpha=E_\alpha(\theta P_* + (1-\theta) P_*') = I_\alpha +
\frac{\alpha}{2 \delta}\sum_{i=1}^{N_b}
\left( (\theta \rho_{*,i}+(1-\theta) \rho_{*,i}')^2 - (\theta
  \rho_{*,i}^2+(1-\theta){\rho_{*,i}'}^2
) \right),
$$
where $I_\alpha = E_\alpha(P_*)=E_\alpha(P_*')$ is the minimum of
\eqref{eq:relaxed}. Since the function $\R \ni x \mapsto x^2 \in \R$ is
strictly convex, we obtain that $\rho_*=\rho_*'$. Therefore, all the
minimizers of \eqref{eq:relaxed} share the same density, hence the same
mean-field Hamiltonian matrix $H_*$. If $P_*$ is a minimizer of
\eqref{eq:relaxed}, it satisfies the first order optimality condition (Euler
inequality)
$$
\Forall P \in {\rm CH}(\Mc_N), \quad \Tr(H_*(P-P_*)) \ge 0,
$$
from which we infer by a classical argument that
\begin{equation}\label{eq:minRC}
P_* = {\mathds 1}_{(-\infty,\mu)}(H_*) + Q_* \quad \mbox{with} \quad 0 \le Q_*
\le 1, \quad \mbox{Ran}(Q_*) \subset  \mbox{Ker}(H_*-\mu), \quad \Tr(P_*)=N,
\end{equation}
for some Fermi level $\mu \in \R$ (the Lagrange multiplier of the constraint $\Tr(P)=N$).
Let $\varepsilon_1 \le \cdots \le \varepsilon_{N_b}$ be the eigenvalues of
$H_*$, counting multiplicities. If $\varepsilon_N < \varepsilon_{N+1}$, then
we necessarily have $P_* = {\mathds 1}_{(-\infty,\varepsilon_N]}(H_*)$, so
that \eqref{eq:relaxed} has a unique minimizer, $P_*$ is on $\Mc_N$ and
therefore is also the unique minimizer of \eqref{eq:FDNLS}, and it satisfies
the strong {\em Aufbau} principle.

Let us now consider the case when $\varepsilon_N =
\varepsilon_{N+1}\eqqcolon\mu$. Since the eigenvalue problem
$H_*\psi=\mu\psi$ is a second-order difference equation
\begin{equation}\label{eq:disc_ee}
\frac{-\psi_{i+1}+2\psi_i-\psi_{i-1}}{2\delta^2} + V_{{\rm eff},i} \psi_i = \mu \psi_i, \quad 1 \le i \le N_b,
\end{equation}
(here and in the sequel we use the convention that $\psi_0=\psi_{N_b}$ and $\psi_{N_b+1}=\psi_1$) with $V_{\rm eff} = V + \alpha\rho_*$, the eigenspace
$\mbox{Ker}(H_*-\mu)$ is at most of dimension 2. We therefore have
$\varepsilon_{N-1} < \varepsilon_N = \varepsilon_{N+1} <
\varepsilon_{N+2}$.

Using the variational characterization of the ground state eigenvalue, we have
\begin{equation}\label{eq:minmax}
\varepsilon_1 = \min_{\psi \in \R^{N_b}, \, \psi^*\psi=1} \psi^* H_*\psi \quad
\mbox{with} \quad \psi^* H_*\psi = \sum_{i=1}^{N_b} \left|
  \frac{\psi_{i+1}-\psi_i}{\delta} \right|^2 + \sum_{i=1}^{N_b} V_{{\rm eff},i}
|\psi_i|^2.
\end{equation}
Since $||x|-|y|| \le |x-y|$ for all $x,y \in \R$ with equality if and only if
$x$ and $y$ have the same sign, we infer from \eqref{eq:minmax}, that all the
entries of a ground state eigenvector of $H_*$ have the same sign. In
particular, two normalized ground state eigenvectors of $H_*$ cannot be
orthogonal. This implies that the ground state eigenvalue of $H_*$ is simple,
i.e. $\varepsilon_1 < \varepsilon_2$. The first statement of Lemma~\ref{lem:NLS}
staightforwardly follows from the results established so far.

To prove the second statement, assume that $N \ge 2$ and that
\eqref{eq:relaxed} has two distinct minimizers $P_*$ and $P_*'$ sharing the same
density. In view of \eqref{eq:minRC},
this can only occur if $\varepsilon_N = \varepsilon_{N+1}\eqqcolon\mu$. Using an
orthonormal basis $(\phi,\psi)$ of $\mbox{Ker}(H_*-\mu)$ consisting of
eigenvectors of $P_*$, we can assume without loss of generality that
\begin{align*}
P_* &= {\mathds 1}_{(-\infty,\mu)}(H_*) + (1-f) \phi\phi^* + f \psi\psi^*, \\
P_*' &= {\mathds 1}_{(-\infty,\mu)}(H_*) + (1-a) \phi\phi^* + a \psi\psi^* + b (\phi\psi^*+\psi\phi^*),
\end{align*}
with $0 \le f \le 1$, $0 \le a \le 1$ and $b^2 \le a(1-a)$. Since $P_*$ and $P_*'$ have the same density, we have for all $1 \le i \le N_b$,
$(1-f) \phi_i^2  + f \psi_i^2 = (1-a) \phi_i^2 + a \psi_i^2+ 2 b \phi_i\psi_i$
, that is $(a-f) \phi_i^2 - 2 b \phi_i\psi_i - (a-f) \psi_i^2=0$.

If $a=f$, then $b \neq 0$ since $P_*\neq P_*'$ by assumption, so that
$\phi_i\psi_i = 0$ for all $1 \le i \le N_b$. From \eqref{eq:disc_ee}, we see
that it is not possible to have $\psi_i=\psi_{i+1}=0$ (otherwise, $\psi$ would
be identically equal to zero), and the same holds true for $\phi$. Therefore,
$N_b$ must be even, and either all the odd entries of $\phi$ and all the even
entries $\psi$ must vanish, or the other way round.  We then infer from
\eqref{eq:disc_ee} that this implies that
$\phi_{i+2}+\phi_i=\psi_{i+2}+\psi_i=0$ for all $1 \le i \le N_b$, and
that all the entries of $V_{\rm eff}$
are equal to $\mu-\delta^{-2}$. This implies that $N_b \in 4 \N^\ast$ and
that the states $\phi$ and $\psi$ are given by $\phi_{2i}=c(-1)^i$,
$\phi_{2i+1}=0$, $\psi_{2i}=0$, $\psi_{2i+1}=c'(-1)^i$ for all $1 \le
i \le N_b$, where $c$ and $c'$ are normalization constants. By
explicit diagonalization of the matrix $H_{*} = H(0) + (\mu - \delta^{-2})
I_{N_{b}}$ one can check that the states  $\phi$ and $\psi$ are therefore
those spanning the two-dimensional space associated to the two-fold
degenerate eigenvalues $\varepsilon_{N_b/2} = \varepsilon_{1+N_b/2}$ of $H_*$. This is only possible if
$N=N_b/2$. The case $\alpha=f$ can thus be excluded for $2 \le N \le N_b$, with $N_b \ne 2N$ if $N_b \in 4\N^\ast$.

If $a \neq f$, we have for all $1 \le i \le N_b$,
$\phi_i^2 - 2 \gamma \phi_i\psi_i - \psi_i^2=0$, for
$\gamma=\frac{b}{a-f}$, and up to replacing $\psi$ with $-\psi$, we
can assume without loss of generality that $\gamma \ge 0$. Denoting by
$C_\pm\coloneqq\gamma \pm \sqrt{1+\gamma^2}$ the roots of the polynomial
$x^2-2\gamma x +1$, with $C_+C_-=-1$, we obtain that for each
$1 \le i \le N_b$, either $\phi_i= C_+ \psi_i$ or $\phi_i= C_- \psi_i$. Using
the discrete Schr\"odinger equation~\eqref{eq:disc_ee} satisfied by both $\phi$
and $\psi$, we see that if $\phi_i= C_+ \psi_i$ and
$\phi_{i+1}= C_+ \psi_{i+1}$ for some $1 \le i \le N_b$, then  $\phi= C_+ \psi$,
and likewise if $C_+$ is replaced by $C_-$. This is impossible since $\phi$
and $\psi$ are orthonormal. Therefore, we must have
$\phi_{2i}= C_+ \psi_{2i}$ and $\phi_{2i+1}=C_-\psi_{2i+1}$ (or the other way
around), and $N_{b}$ must be even. Using again~\eqref{eq:disc_ee}, this leads to $\phi_{i+2}+\phi_i=0$ and
$\psi_{i+2}+\psi_i=0$ for all $1 \le i \le N_b$ and therefore, as in the
previous case, that $N_b \in 4 \N^\ast$, $N_b=2N$, that all the entries of $V_{\rm eff}$ are equal and that
$\phi$ and $\psi$ span the two-dimensional space associated to the two-fold
degenerate eigenvalues $\varepsilon_{N} = \varepsilon_{N+1}$ of $H_*$.

This proves that for $2 \le N \le N_b$, with $N_b \ne 2N$ if $N_b \in 4 \N^\ast$, \eqref{eq:relaxed} has a unique
minimizer $P_*$. If $P_* \in \Mc_N$, it is of course also the unique minimizer
of  \eqref{eq:FDNLS}, and $P_*$ satisfies the {\em Aufbau} principle.
Conversely, if $P_*' \in \Mc_N$ is a local minimizer of \eqref{eq:FDNLS}
satisfying the {\em Aufbau} principle, we have
$$
\Forall P \in {\rm CH}(\Mc_N), \quad \Tr(H(P_*')(P-P_*')) \ge 0,
$$
which means that $P'_*$ is a solution to the Euler inequality for \eqref{eq:relaxed}, and therefore a global minimizer of this convex problem. Since the minimizer $P_*$ of \eqref{eq:relaxed} is unique, we finally obtain that if $P_* \notin \Mc_N$, then none of the local minimizers of \eqref{eq:FDNLS} satisfies the {\em Aufbau} principle.

\section{Acknowledgement}
The authors would like to thank Michael F. Herbst and Sami Siraj-Dine for
fruitful discussions. This project has received funding from the European
Research Council (ERC) under the European Union's Horizon 2020 research and
innovation programme (grant agreement No 810367).

\nocite{*}
\bibliography{./biblio}
\bibliographystyle{abbrv}

\end{document}